\documentclass[11pt]{article}
\usepackage[top=1in, bottom=1in, left=1in, right=1in]{geometry}

\usepackage[linktocpage,colorlinks,linkcolor=blue,anchorcolor=blue,citecolor=blue,urlcolor=blue,pagebackref]{hyperref}
\usepackage{euscript,mdframed}
\usepackage{microtype,todonotes,relsize}
\usepackage{amsmath,amsthm}
\usepackage{algpseudocode}

\usepackage{accents}
\usepackage{comment,url,graphicx,relsize}
\usepackage{amssymb,amsfonts,amsmath,amsthm,amscd,dsfont,mathrsfs,mathtools,nicefrac, bm}
\usepackage{float,psfrag,epsfig,color,xcolor,url}
\usepackage{epstopdf,bbm,mathtools,enumitem}
\usepackage{subfigure}
\usepackage{algorithm}
\usepackage{tablefootnote}
\graphicspath{{Plots/}}
\usepackage{diagbox}
\usepackage{tikz}
\usetikzlibrary{calc,patterns,angles,quotes}
\usepackage{makecell}
\usepackage{multirow}
\usepackage{booktabs}
\usepackage[authoryear,round]{natbib}
\newtheorem{theorem}{Theorem}[section]
\newtheorem{lemma}{Lemma}[section]

\newtheorem{assumption}{Assumption}[section]
\newtheorem{remark}{Remark}[section]

\newtheorem{definition}{Definition}[section]

\newcommand{\argmin}{\operatorname{argmin}}

\newcommand{\inner}[2]{\langle{#1},{#2}\rangle}

\newcommand{\T}{\operatorname{T}}
\newcommand{\E}{\mathbb{E}}
\newcommand{\dist}{\mathrm{dist}}
\newcommand{\Exp}{\mathtt{Exp}}
\newcommand{\D}{\operatorname{D}}
\newcommand{\retr}{\mathtt{Retr}}

\newcommand{\grad}{\mathtt{grad}}
\newcommand{\col}[1]{{\rm col}\left\{#1\right\}}
\newcommand{\f}{\mathbf{f}}
\newcommand{\bx}{\mathbf{x}}
\newcommand{\bz}{\mathbf{z}}
\newcommand{\proj}[1]{\mathcal{P}_{\mathcal{M}}\left(#1\right)}
\newcommand{\blambda}{\mathbf{\Lambda}}
\newcommand{\R}{\mathbb{R}}

\newcommand{\cD}{\mathcal{D}}

\newcommand{\cG}{\mathcal{G}}

\newcommand{\cM}{\mathcal{M}}

\newcommand{\cP}{\mathcal{P}}

\newcommand{\bG}{\mathbf{G}}

\title{Federated Learning on Riemannian Manifolds: \\ A Gradient-Free Projection-Based Approach}

\author{
{Hongye Wang} \thanks{School of Information Management and Engineering, Shanghai University of Finance and Economics. \texttt{ishongyewang@gmail.com}}
\and
{Zhaoye Pan} \thanks{School of Information Management and Engineering, Shanghai University of Finance and Economics. \texttt{iszhaoyepan@gmail.com}}
\and
{Chang He} \thanks{School of Information Management and Engineering, Shanghai University of Finance and Economics; Department of Industrial and System Engineering, University of Minnesota. \texttt{ischanghe@gmail.com}}
\and
{Jiaxiang Li} \thanks{Department of Electrical and Computer Engineering, University of Minnesota.  \texttt{li003755@umn.edu}}
\and 
{Bo Jiang} \thanks{School of Information Management and Engineering, Shanghai University of Finance and Economics. \texttt{isyebojiang@gmail.com}}
}

\begin{document}

\maketitle

\begin{abstract}
Federated learning (FL) has emerged as a powerful paradigm for collaborative model training across distributed clients while preserving data privacy. However, existing FL algorithms predominantly focus on unconstrained optimization problems with exact gradient information, limiting its applicability in scenarios where only noisy function evaluations are accessible or where model parameters are constrained. To address these challenges, we propose a novel zeroth-order projection-based algorithm on Riemannian manifolds for FL. By leveraging the projection operator, we introduce a computationally efficient zeroth-order Riemannian gradient estimator. Unlike existing estimators, ours requires only a simple Euclidean random perturbation, eliminating the need to sample random vectors in the tangent space, thus reducing computational cost. Theoretically, we first prove the approximation properties of the estimator and then establish the sublinear convergence of the proposed algorithm, matching the rate of its first-order counterpart. Numerically, we first assess the efficiency of our estimator using kernel principal component analysis. Furthermore, we apply the proposed algorithm to two real-world scenarios: zeroth-order attacks on deep neural networks and low-rank neural network training to validate the theoretical findings.
\end{abstract}

\section{Introduction}\label{section.intro}
Federated learning (FL) \cite{mcmahan2017communication, letaief2021edge} is a promising paradigm for large-scale machine learning that has attracted increasing attention in recent years. In scenarios where data and computational resources are distributed across diverse clients, such as phones, sensors, and other devices, FL facilitates collaborative model training \cite{kairouz2021advances} without requiring the exchange of local data. This framework effectively reduces communication overhead in distributed and parallel environments. Among the most well-known FL algorithms are FedAvg \cite{konevcny2016federated,mcmahan2017communication,stich2018local} and SCAFFOLD \cite{karimireddy2020scaffold}. Numerous variants, such as STEM \cite{khanduri2021stem}, MIME \cite{karimireddy2020mime}, and CE-LSGD \cite{patel2022towards}, have been proposed to improve convergence rates and stability.

Existing FL algorithms mentioned above mainly focus on unconstrained problems, and the derivative information is readily available or inexpensive to compute. In many applications, only noisy function evaluations are accessible, and the model parameters are subject to complicated constraints. One strategy for dealing with these issues is approaching them from the perspective of \textit{zeroth-order optimization on Riemannian manifolds} \cite{absil2008optimization,boumal2023introduction,li2023zeroth}. Mathematically, we study FL problems in a gradient-free manner, formulated as:
\begin{equation}\label{eq.main}
\begin{aligned}
    \min_{x \in \cM} \ &f(x) \triangleq \frac{1}{n}\sum_{i=1}^n f_i(x), \\
    &f_i(x) \triangleq \E_{\xi_i \sim \mathcal{D}_i} \left[F_i\left(x, \xi_i\right)\right],
\end{aligned}
\end{equation}
where $n$ is the number of clients, and $x$ denotes the model parameters constrained to a Riemannian submanifold $\mathcal{M}$. The local loss $f_i(\cdot)$, associated with client $i$, is smooth but non-convex and is defined as the expectation over $\xi_i \sim \mathcal{D}_i$. Furthermore, we only have access to noisy evaluations of the function value, rather than the exact Riemannian gradient of each local loss.

Currently, all FL algorithms on Riemannian manifolds \cite{li2022federated, huang2024federated, zhang2024nonconvex, huang2024riemannian, xiao2025riemannian} rely on exact derivative information. This makes them inapplicable in scenarios such as robotic control \cite{yuan2019bayesian}, zeroth-order attacks on deep neural networks \cite{tu2019autozoom, li2023zeroth}, or topological dimension reduction \cite{rabadan2019topological}. One possible strategy, however, is to directly integrate zeroth-order Riemannian optimization techniques into the FL framework. While such algorithms may be effective, they impose a severe computational burden. First, existing zeroth-order Riemannian estimators require sampling tangent random vectors or computing the coordinate basis of the tangent space \cite{li2023zeroth,li2023stochastic,he2024riemannian,wang2021greene}, which is non-trivial. Second, the server requires intricate geometric operators, such as exponential and inverse exponential maps and parallel transport, to calculate the tangent space consensus step \cite{li2022federated, huang2024federated}. For instance, in the case of the Stiefel manifold, each client first generates a Gaussian random vector in Euclidean space and then transforms it into a tangent space vector to construct the estimator. After collecting the zeroth-order estimators from all clients, the server solves a nonlinear matrix equation to compute the average. Both steps introduce significant overhead, making the existing approach inefficient for large-scale FL.

\subsection{Related Works}\label{appendix.literature review}
\paragraph{Zeroth-order Riemannian optimization.} Riemannian zeroth-order algorithms typically involve two key steps: constructing Riemannian zeroth-order estimators and integrating them with standard optimization algorithms, such as Riemannian gradient descent. Using noisy evaluations of the objective function, \cite{li2023stochastic} developed randomized zeroth-order estimators for the Riemannian gradient and Hessian, extending the Gaussian smoothing technique \cite{nesterov2017random, balasubramanian2022zeroth} to Riemannian manifolds. Subsequently, \cite{wang2021greene} proposed an alternative zeroth-order gradient estimator based on the Greene–Wu convolution over Riemannian manifolds, demonstrating superior approximation quality compared to the approach by \cite{li2023stochastic}. In terms of Riemannian zeroth-order algorithms, \cite{li2023stochastic} studied several zeroth-order algorithms for stochastic Riemannian optimization, providing the first complexity results. Later, they improved sample complexities by introducing zeroth-order Riemannian averaging stochastic approximation algorithms in \cite{li2023zeroth}. For deterministic objective function, \cite{he2024riemannian} proposed an accelerated zeroth-order Riemannian algorithm based on the coordinate-wise estimator. \cite{fan2021learning} proposed a Riemannian meta-optimization method that learns a gradient-free optimizer but lacks theoretical guarantees. Additionally, \cite{maass2022tracking} investigated the application of zeroth-order algorithms in Riemannian online learning.

\paragraph{Riemannian federated learning.} \cite{li2022federated} developed the first Riemannian FL framework and introduced Riemannian federated SVRG, where the server maps local models onto a tangent space, computes their average, and retracts the average back to the manifold. Subsequently, \cite{huang2024federated} proposed a generic private FL framework on Riemannian manifolds based on differential privacy, analyzing the privacy guarantees while establishing convergence properties. \cite{zhang2024nonconvex} leveraged the projection operator to design a Riemannian FL algorithm, improving computational efficiency. Very recently, \cite{huang2024riemannian} provided the first intrinsic analysis of Riemannian FL algorithm, extending beyond operations in Euclidean embedded submanifolds.

\subsection{Main Contribution} 
To address the computational burden, in this paper, we propose a zeroth-order \textit{projection-based} Riemannian algorithm for FL. Our main contributions are:

Firstly, we develop a new zeroth-order Riemannian gradient estimator based on the projection operator \cite{absil2012projection}. The estimator does not require sampling random vectors in the tangent space. Instead, it only needs Euclidean random vectors, similar to the zeroth-order estimators \cite{nesterov2017random, balasubramanian2022zeroth, ghadimi2013stochastic} in Euclidean space. As a result, our estimator is computationally more efficient than those in \cite{li2023zeroth, wang2021greene, he2024riemannian}. Theoretically, by leveraging the proximal smoothness of compact smooth submanifolds, we establish the approximation properties of the estimator, thereby ensuring its applicability in zeroth-order algorithms.

 Subsequently, we incorporate the zeroth-order estimator into the FL framework. To mitigate client drift, correction terms are introduced during the local updates on each client. At the server, we employ the projection operator for the tangent space consensus step, eliminating the need for exponential maps, inverse exponential maps, and parallel transport. By appropriately selecting the smoothing parameter, we establish sublinear convergence to an approximate first-order optimal solution under heterogeneous client data, matching the rate of its first-order counterpart \cite{zhang2024nonconvex}. Furthermore, our algorithm achieves the \textit{linear speedup} property, with the convergence rate improving proportionally to the number of participating clients.

 Finally, numerical experiments are conducted to validate the effectiveness of our proposed algorithm. We first assess the efficiency of the proposed zeroth-order Riemannian estimator using kPCA, demonstrating its computational advantages over the existing estimator. Next, we apply the algorithm to two practical applications: zeroth-order attacks on deep neural networks and low-rank neural network training, further supporting our theoretical findings.

\section{Preliminaries: optimization on manifolds}\label{section.pre}
In this section, elementary mathematical tools for Riemannian optimization are presented. For more details, we refer readers to \cite{absil2008optimization,boumal2023introduction}. Throughout this paper, we consider compact smooth Riemannian submanifolds embedded in the Euclidean space $\mathbb{R}^{p \times r}$, e.g., the Stiefel manifold and the Oblique manifold. The matrix $x \in \mathbb{R}^{p \times r}$ is denoted using lowercase letters. Each $x \in \cM$ is associated with a real vector space $\T_x\cM$, referred to as the tangent space at $x$. The normal space, denoted as $\mathrm{N}_x \mathcal{M}$, is orthogonal to the tangent space. For example, the Stiefel manifold $\mathcal{M}=\operatorname{St}(p, r)\triangleq\left\{x \in \mathbb{R}^{p \times r}: x^{\top} x=I_r\right\}$ is a smooth submanifold, and the tangent space at point $x$ is given by $\T_x \mathcal{M}=\left\{y \in \mathbb{R}^{p \times r}: x^{\top} y+y^{\top} x=0\right\}$. We use $\langle x, y\rangle=\operatorname{Tr}\left(x^{\top} y\right)$ to denote the Euclidean inner product of two matrices $x, y$, and the Riemannian metric is induced from the Euclidean inner product. We use $\|x\|$ to denote the Frobenius norm of $x$. For the local loss $f_i(\cdot)$ at client $i$, the Riemannian gradient $\grad f_i(x)$ at $x \in \cM$ is the unique vector in $\T_x\cM$ that satisfies $\D f_i(x)[s] = \inner{\grad f_i(x)}{s}$ for all $s \in \T_x\cM$, where $\D f_i(x)[s]$ is the directional derivative of $f_i$ at $x$ along $s$.

To optimize over Riemannian manifolds, a key ingredient is the retraction—a mapping enabling movement along the manifold from a point $x$ in the direction of a tangent vector $s \in \T_x\cM$. The exponential map $\Exp_x(s)$ is a typical example of a retraction map. For the Stiefel manifold $\operatorname{St}(p, r)$, the polar decomposition $\retr_x(s) = (x+s)(I_r + s^\top s)^{-1/2}$ serves as a retraction. In this paper, we utilize the projection operator $\cP_{\mathcal{X}}$, where $\mathcal{X} \subseteq \mathbb{R}^{p \times r}$, defined by
\begin{equation}\label{eq.projection operator}
    \cP_{\mathcal{X}}(s)\in \argmin_{x\in\mathcal{X}}\frac{1}{2}\|s-x\|^2, s \in \mathbb{R}^{p \times r}
\end{equation}
as a substitute for the retraction. It is well known that the projection exists when the set $\mathcal{X}$ is closed, but it may not be unique. When the projection set is a manifold (i.e., $\mathcal{X} = \cM$), the existence of the projection can be guaranteed locally without any further assumptions, as a manifold is always locally closed. Specifically, $\cP_{\cM}(x + s)$ can be regarded as a retraction when $x \in \cM$ and $s \in \T_x\cM$. The following concept quantitatively characterizes the uniqueness of the projection onto a manifold.
\begin{definition}[$\hat{\gamma}$-proximal smoothness of $\cM$]
    For any $\hat{\gamma}>0$, we define the $\hat{\gamma}$-tube around $\cM$ as $ U_{\cM}(\hat{\gamma})\triangleq\{x : \dist(x, \cM) < \hat{\gamma}\}$, where $\dist(x,\cM)\triangleq \min_{u\in \cM}\|u-x\|$ is the Eulidean distance between $x$ and $\cM$. We say that $\cM$ is $\hat{\gamma}$-proximally smooth if the projection operator $\cP_{\cM}(x)$ is unique whenever $x\in U_{\cM}(\hat{\gamma})$. For example, Stiefel manifold is $1$-proximally smooth.
\end{definition}

Any compact smooth submanifold $\mathcal{M}$ embedded in $\mathbb{R}^{p \times r}$ is a proximally smooth set \cite{davis2025stochastic}. Without loss of generality, we make the following assumption on the Riemannian manifold, which is widely used in \cite{zhang2024nonconvex,deng2023decentralized}.
\begin{assumption}\label{assum.proximal-smooth}
    The Riemannian submanifold $\cM$ is $2\gamma$-proximally smooth.
\end{assumption}
It immediately implies the following properties, and the proofs can be found in \cite{davis2025stochastic} and \cite{zhang2024nonconvex}.
\begin{lemma}\label{lemma.properties of proximal smoothness}
    Suppose Assumption \ref{assum.proximal-smooth} holds. Let $\overline{U}_{\cM}(\gamma)\triangleq\{x : \dist(x, \cM) \le \gamma\}$ be the closure of $U_{\cM}(\gamma)$, then the non-expansive property holds, i.e., 
    \begin{equation}\label{eq.lips=2}
        \|\cP_{\cM}(x) - \cP_{\cM}(y)\|\le 2\|x-y\|, \ \forall x,y \in \overline{U}_{\cM}(\gamma).
    \end{equation}
    For any $x\in \cM$ and $v\in \operatorname{N}_x\cM$, it holds that 
    \begin{equation}\label{eq.normal}
        \langle v, y-x\rangle\le \frac{\|v\|}{4\gamma}\|y-x\|^2, \ \forall y\in \cM.
    \end{equation}
    Furthermore, there exists a constant $M>0$ such that for any $x\in\cM$ and $u \in \mathbb{R}^{p \times r}$:
    \begin{equation}\label{eq.project Lips}
        \|\cP_{\cM}(x+u)-x\|\le M\|u\|.
    \end{equation}
\end{lemma}

We close this section by presenting some mild assumptions. As we are only accessible to the noisy evaluation of the function value $F(\cdot, \xi_i)$, we make the following assumption on the zeroth-order oracle.
\begin{assumption}\label{assum.unbiased}
    For any $x \in \mathbb{R}^{p \times r}$ and client $i$, the zeroth-order oracle outputs an estimator $F_i(x, \xi_i)$ of $f_i(x)$ such that $\mathbb{E}_{\xi_i}\left[ F_i\left(x, \xi_i\right)\right]= f_i(x)$, $\mathbb{E}_{\xi_i}\left[\grad F_i\left(x, \xi_i\right)\right]=\grad f_i(x)$ and $\mathbb{E}_{\xi_i}\left[\left\|\grad F_i\left(x, \xi_i\right)-\grad f_i(x)\right\|^2\right] \leq \sigma^2$.
\end{assumption}
It is worth noting that in the above assumption, we do not observe $\grad F_i(x, \xi_i)$ and we just assume that it is an unbiased estimator of gradient of $f_i$ and its variance is bounded. Furthermore, we make a similar smoothness assumption about the noisy function $F_i(x,\xi_i)$ as those in \cite{zhang2024nonconvex,deng2023decentralized,chen2021decentralized}.

\begin{assumption}\label{assum.Lipschitz}
    For each $F_i$, there exists a constant $\chi > 0$ such that $\max_{x \in \cM} \|\nabla F_i\left(x, \xi_i\right)\| \le \chi$ holds almost surely. Moreover, each $F_i$ is $l_i$-smooth on the convex hull of $\cM$ in the Euclidean sense, i.e., for any $x,y\in \mathrm{conv}(\cM)$, $ \|\nabla F_i(x, \xi) - \nabla F_i(y, \xi)\| \le l_i \|x - y\|$ holds almost surely.
\end{assumption}

\section{Zeroth-Order Riemannian Estimator}\label{section.estimator}
We begin by introducing our zeroth-order projection-based Riemannian gradient estimator. For clarity, we temporarily omit the subscript $i$ in \eqref{eq.main}, and focus on the smooth stochastic function $f = \mathbb{E}_{\xi \sim \mathcal{D}}\left[F(x,\xi)\right]$ in this section. Let us recall the existing zeroth-order Riemannian estimator at the point $x \in \cM$ in \cite{li2023stochastic}:
\begin{equation}\label{eq.existing gradient estimator}
      G^{\mathtt{R}}_{\mu}(x) = \frac{1}{m} \sum_{j=1}^m \frac{F(\retr_x(\mu u_j), \xi_{j}) - F(x, \xi_{j})}{\mu} u_{j},  
\end{equation}
where $\mu > 0$ is called the smoothing parameter, $\xi_j$ are independent copies of $\xi$, and $u_j$ are independent standard Gaussian random vectors supported on
$\T_x\cM$. The design of \eqref{eq.existing gradient estimator} is intuitive. A direct perturbation at the point $x$, i.e., $x + \mu u_j$, does not lie on the manifold due to the manifold's nonlinearity. To address this issue, \cite{li2023stochastic} introduces perturbations within the tangent space $\T_x\cM$, which is a vector space, and then employs a retraction to map the perturbations back to the manifold, i.e., $\retr_{x}(\mu u_j)$, thereby mimicking the perturbation process used in Euclidean space. However, generating random vectors in the tangent space is non-trivial. For submanifolds, this can be achieved by using an orthogonal projection matrix to project a standard Gaussian random vector onto the tangent space, since the tangent space is a subspace of the Euclidean space. For general Riemannian manifolds, \cite{utpala2023improved}\footnote{This paper does not focus on Riemannian zeroth-order optimization but instead addresses differentially private Riemannian optimization, which also necessitates generating random vectors in the tangent space.} proposes a sampling approach based on isometric transportation. 

Given the estimator \eqref{eq.existing gradient estimator}, generating tangent random vectors is unavoidable, as the retraction operates exclusively on tangent vectors. To eliminate the need for tangent random vectors, our approach employs the projection operator \eqref{eq.projection operator}. Notably, the domain of the projection operator is the entire space $\mathbb{R}^{p \times r}$, allowing direct perturbations at point $x$ to remain valid. Moreover, the properties of proximal smoothness (Lemma \ref{lemma.properties of proximal smoothness}) and Lipschitz continuity (Assumption \ref{assum.Lipschitz}) theoretically guarantee a good approximation, provided the perturbation radius is sufficiently small. By combining these insights, we introduce our estimator as follows:
\begin{equation}\label{eq.gradient estimator}
    G^{\mathtt{P}}_{\mu}(x) = \frac{pr}{m} \sum_{j=1}^m \frac{F(\cP_{\cM}(x+\mu u_j), \xi_{j}) - F(x, \xi_{j})}{\mu} u_{j}. 
\end{equation}
Here, $u_{j} \sim \mathbb{S}$ are independently sampled from the uniform distribution over the unit sphere in $\R^{p\times r}$ for all $j = 1,\ldots,m$. For the estimator \eqref{eq.gradient estimator}, the vectors $u_j$ no longer need to remain in the tangent space; simple Euclidean perturbations are sufficient. Next, we introduce several key inequalities that serve as key components in the analysis.
\begin{lemma}\label{lemma.lispchitz continuity over Riemannian gradient}
   Suppose the stochastic function is $l$-smooth on the convex hull of $\cM$ in the Euclidean sense and $\max_{x \in \cM} \|\nabla F\left(x, \xi\right)\| \le \chi$
    holds almost surely. Then, there exists a constant $L > 0$ such that for any $x,y\in\cM$, 
   \begin{equation}\label{eq.Lips inequality}   
       \begin{aligned}
         F(y,\xi) - F(x,\xi) 
           \le \inner{\grad F(x,\xi)}{y-x} + \frac{L}{2}\|x-y\|^2, 
       \end{aligned}
   \end{equation}
   and
   \begin{equation}\label{eq.grad Lips}
     \|\grad F(x,\xi)-\grad F(y,\xi)\| \le L\|x-y\|,  
   \end{equation}
   holds almost surely.
       
\end{lemma}
The intuition behind the above lemma is that the Euclidean gradient, $\nabla F(x,\xi)$, can be decomposed into the Riemannian gradient $\grad F(x,\xi)$, and a residual component lying in the normal space $\operatorname{N}_x\cM$, which can be bounded using inequality \eqref{eq.normal}. This allows us to derive Lipschitz-type inequalities for the Riemannian gradient based on Assumption \ref{assum.Lipschitz}.

Next, we explain why Euclidean perturbation is effective. Since $u_1, \ldots, u_m$ in \eqref{eq.gradient estimator} are independent and identically distributed random vectors, we analyze a single random vector $u$ here for simplicity. The small Euclidean perturbation $\mu u$ can be decomposed as follows:
$$
    \mu u = \underbrace{\cP_{\cM}(x+ \mu u) - x}_{\triangleq \rm (I)} + \underbrace{x + \mu u - \cP_{\cM}(x+ \mu u)}_{\triangleq \rm (II)}.
$$
Term $\rm (I)$ is exactly the gap between the original point $x$ and the projection of the perturbed point $x + \mu u$ onto the manifold, while term $\rm (II)$ can be shown to lie in the normal space at the projected point $\cP_{\cM}(x + \mu u)$. This holds whenever $\mu$ is sufficiently small (Lemma \ref{lemma.projection normal property}). A visual illustration is provided in Fig.  \ref{fig:projection}. Hence we may use the Lipschitz properties in Lemma \ref{lemma.lispchitz continuity over Riemannian gradient} to handle the Euclidean perturbation after decomposition. Based on these observations, the approximation capability of the gradient estimator \eqref{eq.gradient estimator} is analogous to that of \eqref{eq.existing gradient estimator}, as demonstrated in \cite{li2023stochastic}.
\begin{figure}[H]
    \centering
    \includegraphics[scale=0.35]{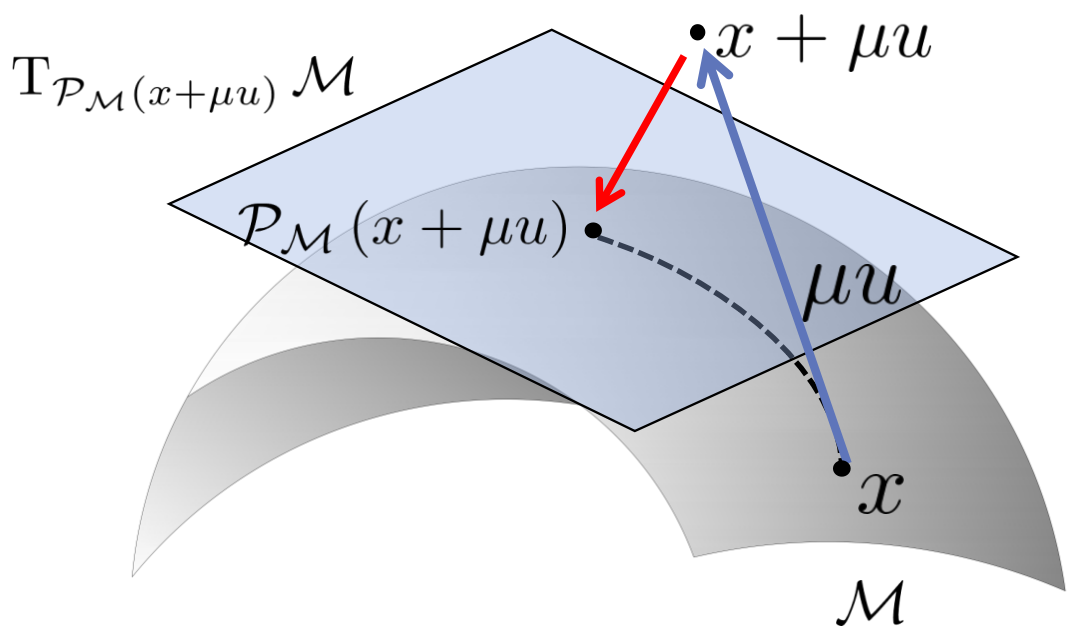}
    \caption{Euclidean perturbation}
    \label{fig:projection}
\end{figure}
\begin{remark}
Previous Riemannian zeroth-order methods \cite{li2023stochastic, li2023zeroth, maass2022tracking} rely on retractions to construct gradient estimators, which requires sampling random vectors in the tangent space and leads to additional overhead. By leveraging the fact that projections can operate on arbitrary ambient vectors, we avoid the sampling procedure in the tangent space. For example, on the Stiefel manifold $\mathrm{St}(n, p)$, the sampling procedure in the $\T_x\mathrm{St}(n, p)$ requires $\mathcal{O}(np^2)$ computational cost.
\end{remark}

We now establish some key properties of the zeroth-order Riemannian gradient estimator $G^{\mathtt{P}}_{\mu}(x)$ defined in \eqref{eq.gradient estimator}. These results will play a crucial role in the subsequent convergence analysis. We begin by showing that the estimator is uniformly bounded under mild assumptions.

\begin{lemma}\label{lemma.uniform boundness}
    Suppose Assumption \ref{assum.proximal-smooth}, \ref{assum.Lipschitz} holds. For the Riemannian gradient estimator \eqref{eq.gradient estimator} with $\mu \le \gamma$, there exists a constant $\chi_G = \mathcal{O}(pr)$ such that $\|G^{\mathtt{P}}_{\mu}(x)\| \le \chi_G$.
\end{lemma}

Next, we show that the expectation of $G^{\mathtt{P}}_{\mu}(x)$ is a good approximation of the true Riemannian gradient \( \grad f(x) \), with an error that depends on the smoothing parameter $\mu$. 

\begin{lemma}\label{lemma.first-order moment}
    Suppose Assumptions \ref{assum.proximal-smooth}, \ref{assum.unbiased}, \ref{assum.Lipschitz} holds. For the Riemannian gradient estimator \eqref{eq.gradient estimator} with $\mu \le \gamma$, there exists a constant $\chi_f = \mathcal{O}(pr)$ such that $\left\|\E \left[G^{\mathtt{P}}_{\mu}(x)\right] - \grad f(x) \right\| \le \chi_f\mu$. The expectation $\mathbb{E}$ is taken with respect to all random vectors $U=\left\{u_{1}, \ldots, u_{m}\right\}$ and $\Xi=\left\{\xi_{1}, \ldots, \xi_{m}\right\}$.
\end{lemma}

Finally, we characterize the second-order moment of the estimator. The result below bounds the mean squared deviation of the estimator from the true gradient in terms of both the smoothing parameter $\mu$ and the sample size $m$.

\begin{lemma}\label{lemma.second-order moment}
    Suppose Assumptions \ref{assum.proximal-smooth}, \ref{assum.unbiased}, \ref{assum.Lipschitz} holds. For the Riemannian gradient estimator \eqref{eq.gradient estimator} with $\mu \le \gamma$, there exists constants $\chi_1 = \mathcal{O}\left(p^2r^2\right), \chi_2 = \mathcal{O}\left(p^2r^2 + \sigma^2\right)$ such that $ \E \left[\left\|G^{\mathtt{P}}_{\mu}(x) - \grad f(x)\right\|^2\right] \le \chi_1\mu^2 + \chi_2/m$. The expectation $\mathbb{E}$ is taken with respect to all random vectors $U=\left\{u_{1}, \ldots, u_{m}\right\}$ and $\Xi=\left\{\xi_{1}, \ldots, \xi_{m}\right\}$.
\end{lemma}

\section{Algorithm Design and Analysis}\label{section.algorithm}
In this section, we integrate the proposed estimator into the FL framework on Riemannian manifolds. Our zeroth-order projection-based Riemannian FL is inspired by recent FL algorithms proposed in \cite{zhang2024composite,zhang2024nonconvex}. We use $k$ as the index for communication rounds and $t$ for local updates, both as superscripts. The subscript $i$ denotes the $i$-th client, and $j$ denotes the $j$-th stochastic sample used to construct the zeroth-order estimator.

\subsection{Algorithm design}
At each communication round $k$, each client $i$ updates two local variables, $\hat{z}_{i}^{k,t}$ and $z_{i}^{k,t}$. Since the data distributions across clients typically differ in practice, $\hat{z}_{i}^{k,t}$ not only aggregates the zeroth-order Riemannian estimators but also incorporates correction terms to mitigate client drift. Additionally, $z_{i, t}^r=\mathcal{P}_{\mathcal{M}}(\hat{z}_{i, t}^r)$ ensures that the estimator is calculated at points on the manifold $\mathcal{M}$. After $\tau$ local updates, the server receives all $\hat{z}_{i}^{k,\tau}$ and computes the average to construct the global model $x^{k+1}$, which is then broadcast to each client $i$. Subsequently, each client $i$ uses $x^{k+1}$ to update the correction term $c_i^{r+1}$ locally. Since the averaging step is computed through the projection operator, the inverse exponential mapping and parallel transport are no longer required, resulting in a lower computational cost compared with \cite{li2022federated}. Notice that $x^k$ does not necessarily remain on the manifold, and we only need to project the final point $x^{K+1}$ onto $\cM$ after $K$ communication rounds. By putting all things together, the detailed implementation of our algorithm is presented in Algorithm \ref{alg.federated}.

\subsection{Algorithm analysis} 
In zeroth-order optimization within Euclidean space, the analysis typically uses a smoothing function as a surrogate, since the estimator serves as an unbiased estimator of the gradient of this surrogate function (see \cite{nesterov2017random, ghadimi2013stochastic}). However, such a smoothing function no longer exists in the Riemannian setting due to the incorporation of the projection operator for constructing the zeroth-order Riemannian estimator. Consequently, we analyze the original function directly, focusing on the optimality gap $f(\cP_{\cM}(x^k)) - f^*$, where $f^*$ denotes the optimal value of problem \eqref{eq.main}. Since the estimator $G^{\mathtt{P}}_{\mu}(x)$ does not necessarily lie in the tangent space $\operatorname{T}_x \mathcal{M}$, the Lipschitz-type inequality \eqref{eq.Lips inequality} cannot be directly applied. Instead, the descent property for a Euclidean update direction is established in Lemma \ref{lemma.special}.
\begin{algorithm}[h]
 \caption{Zeroth-order Projection-based  Riemannian Federated Learning}
 \begin{algorithmic}[1]
            \State{\textbf{Input:}} Total communication rouds $K$, local update steps $\tau$, step sizes $\eta$, $\eta_g$, $\tilde{\eta} = \eta \eta_g\tau$, initial global model $x^1$, and initial correction terms $c_i^1 = 0$ for all $i = 1, \ldots, n$
            
  \For{$k=1,2,\cdots,K$}
            \State\textbf{Client $i = 1,\cdots, n$ in parallel:}
            \State Set $\hat{z}_{i}^{k,0} = \cP_{\cM}(x^k)$ and $z_{i}^{k,0} = \cP_{\cM}(x^k)$
            \For{$t = 0,1,\cdots,\tau-1$}
                \State Sample $\xi_{i,j}^{k,t} \sim \mathcal{D}_i$, $j = 1,\ldots,m$ independently;
      \State Sample $u_{i,j}^{k,t} \sim \mathbb{S}$, $j = 1,\ldots,m$ independently;
                \State Construct local estimator $G_{i}^{k,t}$ at $z_{i}^{k,t}$ using \eqref{eq.gradient estimator};
                \State Update $$\hat{z}_{i}^{k,t+1} = \hat{z}_{i}^{k,t} - \eta \left(G_{i}^{k,t}+c_i^k\right);$$
                \State Update $$z_{i}^{k,t+1} = \cP_{\cM}(\hat{z}_{i}^{k,t+1});$$
     \EndFor
            \State Send $\hat{z}_{i}^{k,t+1}$ to the server
            \State \textbf{Server:}
            \State Update $$x^{k+1}=\cP_{\cM}(x^k) + \eta_g\left(\frac{1}{n}\sum_{i=1}^n\hat{z}_{i}^{k,\tau}-\cP_{\cM}(x^k)\right)$$
            \State Broadcast $x^{k+1}$ to all the cilents
            \State\textbf{Client $i = 1,\cdots, n$ in parallel:}
            \State Receive $x^{k+1}$ from the server
            \State Update $$c_i^{k+1} = \frac{1}{\eta_g\eta\tau}\left(\cP_{\cM}(x^k)-x^{k+1}\right) - \frac{1}{\tau}\sum_{t = 0}^{\tau - 1}G_{i}^{k,t}$$ 
        \EndFor
        \State\textbf{Output:} $\cP_{\cM}(x^{K+1})$
 \end{algorithmic}
 \label{alg.federated}
\end{algorithm}
\begin{lemma}\label{lemma.special}
        For any $x\in\cM$, $v\in \R^{p\times r}$ and $\eta>0$, let $x^+=\proj{x-\eta v}$, if $x-\eta v\in \overline{U}_{\cM}(\gamma)$, then it follows that for any $z\in\cM$:
        \begin{equation}
            \begin{split}
                f(x^+)\le  &f(z)+\inner{\grad f(x)-v}{x^+-z}-\frac{1}{2\eta}(\|x^+ -x\|^2-\|z-x\|^2) \\
                 &-\left(\frac{1}{2\eta}-\frac{3\|v\|}{4\gamma}\right)\|z-x^+\|^2+\frac{L}{2}\|x^+-x\|^2+\frac{L}{2}\|z-x\|^2. \notag
            \end{split}
        \end{equation}
    \end{lemma}

\vspace{0.2cm}
For the convergence measure, we adopt the first-order suboptimality metric $\|\cG_{\tilde{\eta}}(\proj{x^k})\|$ proposed in \cite{zhang2024nonconvex}, defined as $ \cG_{\tilde{\eta}}(\proj{x^k}) \triangleq (\proj{x^k} - \tilde{x}^{k+1})/\tilde{\eta}$, where $\tilde{x}^{k+1}\triangleq\proj{\proj{x^k}-\tilde{\eta}\grad f(\proj{x^k})}$. Notice that $\tilde{x}^{k+1}$ represents the update of the centralized projected Riemannian gradient descent and is introduced solely for theoretical analysis. This suboptimality metric is valid since $\|\cG_{\tilde{\eta}}(\proj{x^k})\|=0$ if and only if $\|\grad f(\proj{x^k})\|=0$. For a detailed proof, we refer readers to Lemma A.2 in \cite{zhang2024nonconvex}. Since $\cP_{\cM}(\cdot)$ is sufficiently smooth over $\overline{U}_{\cM}(\gamma)$, we introduce $L_{\cP}=\max_{x\in\overline{U}_{\cM}(\gamma)}\|\operatorname{D}^2\proj{x}\|$ as the smoothness constant of its second-order differential, which is used in determining the step size. Now, we provide the theoretical guarantee of Algorithm \ref{alg.federated}.
\begin{theorem}\label{theorem.convergence}
    Suppose Assumptions \ref{assum.proximal-smooth}, \ref{assum.unbiased} and \ref{assum.Lipschitz} hold. If the step sizes satisfy $\eta_g = \sqrt{n}$ and
    \begin{align*}
       \tilde{\eta}\triangleq \eta_g\eta\tau\le \min\left\{\frac{1}{24ML},\frac{\gamma}{6\max\{\chi_G,\chi\}},\frac{1}{\chi L_{\cP}}\right\}, 
    \end{align*}
    then we have 
    \begin{align*}
        \frac{1}
        {K}\sum_{k=1}^K\|\cG_{\tilde{\eta}}(\proj{x^k})\|^2 
        \le \frac{8\Omega^1}{\sqrt{n}\eta\tau K}+\frac{64}{n\tau}\left(\chi_1\mu^2+\frac{\chi_2}{m}\right) 
        +\frac{16(3+n)\left(\chi_G + \chi\right)\chi_f}{n}\mu,    
    \end{align*}
     where $\Omega^1>0$ is a constant related to initialization. Specifically, set the smoothing parameter $\mu = \mathcal{O}(1/p r n\tau K)$, and it follows
     \begin{align*}
        &\frac{1}
        {K}\sum_{k=1}^K\|\cG_{\tilde{\eta}}(\proj{x^k})\|^2 
        = \mathcal{O}\left(\frac{1}{\sqrt{n} \tau K} + \frac{1}{n\tau m}\right).
     \end{align*}
\end{theorem}
\begin{remark}
    By properly setting the value of the smoothing parameter, the convergence rate exhibits an inverse dependency on the number of clients $n$, a property commonly referred to as linear speedup \cite{lian2017can,tang2018d,li2024problem}. Besides, it also indicates that multiple local updates contribute to faster convergence. Compared to the first-order counterpart in \cite{zhang2024nonconvex}, our algorithm utilizes a smaller step size due to the zeroth-order oracle, as the step size choice is inversely proportional to $\chi_G = \mathcal{O}(pr)$.
\end{remark}

\section{Numerical Experiments}
In this section, we conduct three experiments to evaluate the performance of the proposed estimator and algorithm. In the first experiment, we compare estimators \eqref{eq.existing gradient estimator} and \eqref{eq.gradient estimator} by incorporating them into zeroth-order Riemannian gradient descent to validate the computational advantage of our proposed estimator. The remaining two experiments focus on zeroth-order attacks on deep neural networks and low-rank neural network training under different rank constraints. Various parameter configurations are tested to support the theoretical results of the proposed algorithm \ref{alg.federated}.

\subsection{Centralized kPCA}
In this subsection, we assess the computational advantage of the proposed estimator \eqref{eq.gradient estimator}. To isolate its performance, we do not consider the FL framework and instead evaluate the following centralized kernel principal component analysis (kPCA) problem over the Stiefel manifold:
\begin{align*}
\min_{x \in \operatorname{St}(p,r)} \ f(x) = -\frac{1}{2N}\sum_{j=1}^N \operatorname{Tr}\left(x^\top H_j x\right),
\end{align*}
where $\{H_j\}_{j=1}^N$ represents the dataset. The global optimum $f^*$ is precomputed using Pymanopt \cite{townsend2016pymanopt} with the Riemannian trust-region method. To evaluate the computational efficiency of our estimator, we apply the zeroth-order Riemannian gradient descent method to solve the problem. We compare our proposed estimator $G^{\mathtt{p}}_{\mu}(x)$ \eqref{eq.gradient estimator} with the existing Riemannian estimator $G^{\mathtt{R}}_{\mu}(x)$ \eqref{eq.existing gradient estimator}. Our estimator utilizes the projection operator, and the existing one uses polar decomposition and exponential mapping as retractions. The smoothing parameter $\mu$ and batch size $m$ are set identically across all estimators for a fair comparison. 

We conduct experiments on the Digits and Iris datasets from scikit-learn \cite{pedregosa2011scikit}. Fig. \ref{fig:efficiency} presents the objective function value gap over runtime. The results align with our theoretical analysis: leveraging the projection operator eliminates the need for sampling tangent random vectors, significantly reducing computation time.

\begin{figure}[H]
\centering
\subfigure[Iris: $N = 150$, $r = 2$, $\mu = 1e-4$, $m = 100$]{\includegraphics[width=0.4\textwidth]{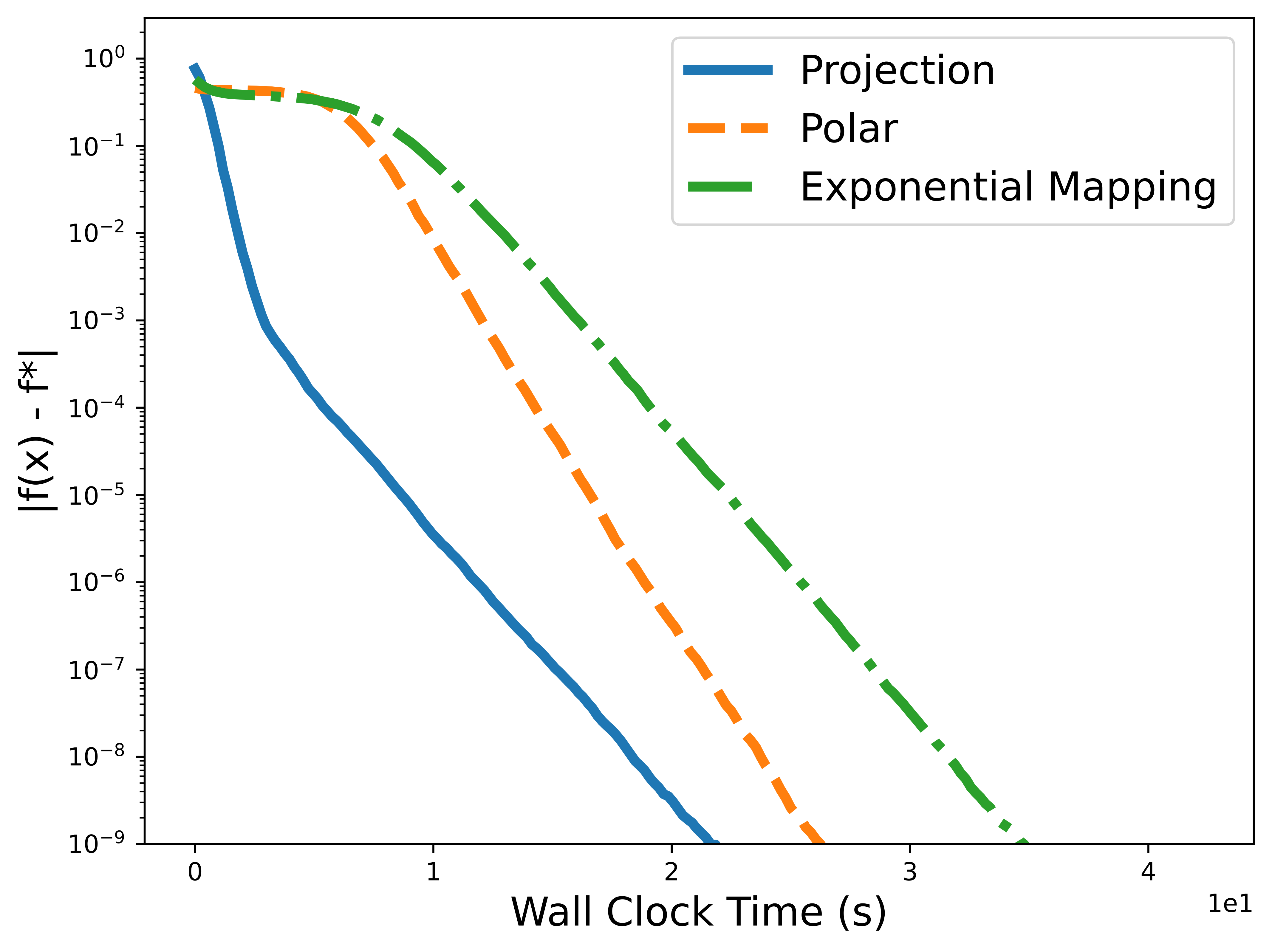}} 
\subfigure[Digits: $N = 200$, $r = 4$, $\mu = 1e-4$, $m = 50$]{\includegraphics[width=0.4\textwidth]{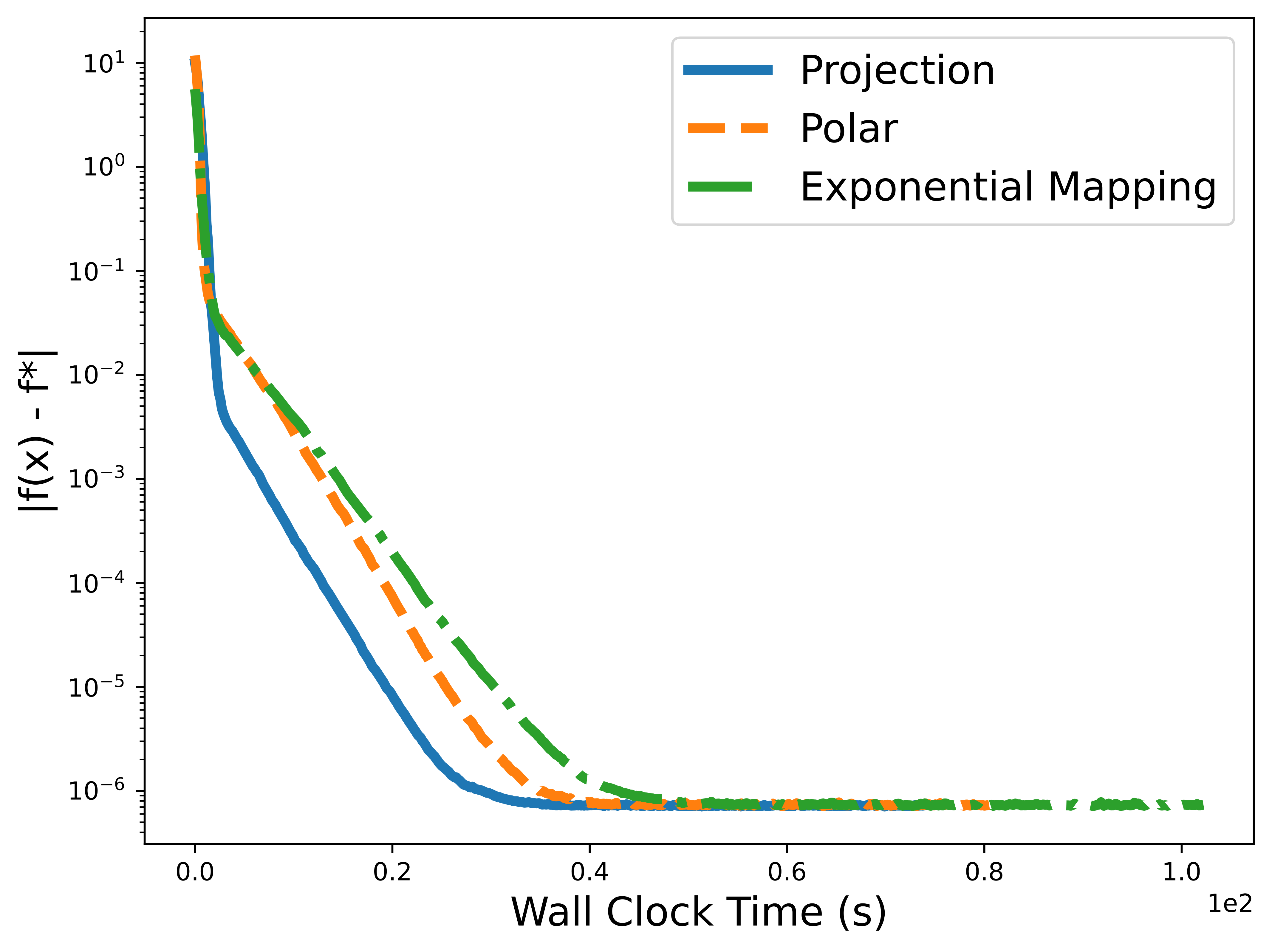}}
\vspace{-10pt}
\caption{Comparison with the existing estimator.}
\label{fig:efficiency}
\end{figure}


\begin{figure*}[h]
\centering
\includegraphics[width=\linewidth]{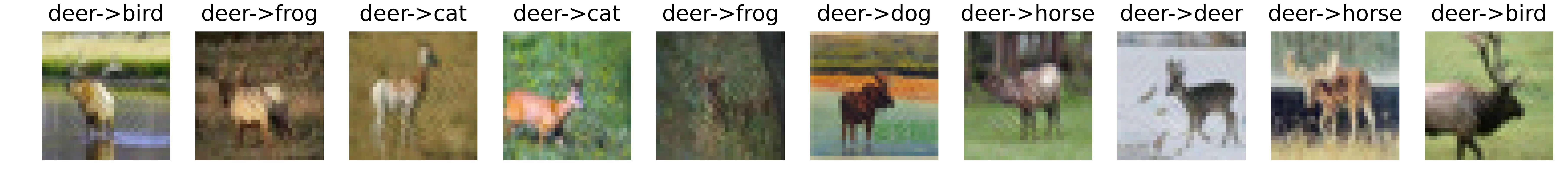} %
\vspace{-0.8cm}
\caption{Example of the proposed attack algorithm on a deer picture. Each subfigure demonstrates the transformation from the original image (left) to the adversarial example (right), with the true label and the predicted label after the attack indicated by the notation $\text{true label} \rightarrow \text{predicted label}$. 
The results highlight the robustness of the attack algorithm in inducing misclassification across diverse inputs.}
\label{fig:attack}
\end{figure*}

\subsection{Zeroth-order attacks on deep neural networks}
Existing black-box attacks on deep neural networks utilize Euclidean zeroth-order optimization algorithms to design adversarial attacks \citep{chen2017zoo,tu2019autozoom}. However, a drawback of existing approaches is that the perturbed examples designed to fool the neural network may not belong to the same domain as the original training data \citep{li2023stochastic}. For instance, while natural images typically lie on a manifold \citep{weinberger2006unsupervised}, conventional perturbations are not restricted to this manifold. 

We consider black-box attacks targeting image classification deep learning models that have been well-trained on publicly available CIFAR-10 dataset, which is widely recognized as a standard benchmark for black-box adversarial attacks in the vision community. \citep{fang2022communication,yi2022zeroth}. The goal of this task is to collaboratively generate a subtle perturbation that remains visually imperceptible to humans while causing the classifier to make incorrect predictions. In our experiment, we assume the manifold is a sphere. 
The sphere manifold naturally enforces the $\ell_2$ distortion constraint $\|\delta\|_2 \leq \epsilon$, which aligns with human visual perception as demonstrated in foundational works \cite{carlini2017towards}. This geometric consistency ensures perturbations remain imperceptible while maximizing attack success. 
As shown in Fig \ref{fig:attack}, a deer image perturbed with small $\ell_2$ distortion successfully misclassifies the target model while preserving visual fidelity.

We adopt the same data-partitioning protocol as in \citet{fang2022communication, li2023stochastic}. A 4-layer CNN pre-trained on CIFAR-10(82\% accuracy) is deployed on each client. The 4992 correctly-classified deer images are sharded without overlap so each client holds a private, label-noise-free subset. All adversarial objectives are optimized with the geodesic update on the $\ell_2$-norm sphere, using batch size $25$, step size $0.001$ and balancing coefficient $c=1$. The convergence results are reported in communication rounds across Figures \ref{fig:blackbox-tau}, \ref{fig:blackbox-n}, and \ref{fig:blackbox-q}. The results support the theoretical findings in Theorem \ref{theorem.convergence}, demonstrating consistent performance across varying local update steps, client numbers, and batch sizes of the gradient estimator. Specifically, we present visualizations of adversarial examples generated by our attack framework, as illustrated in Figure \ref{fig:attack}.

\begin{figure}[H]
\centering
\subfigure[]{\includegraphics[width=0.4\textwidth]{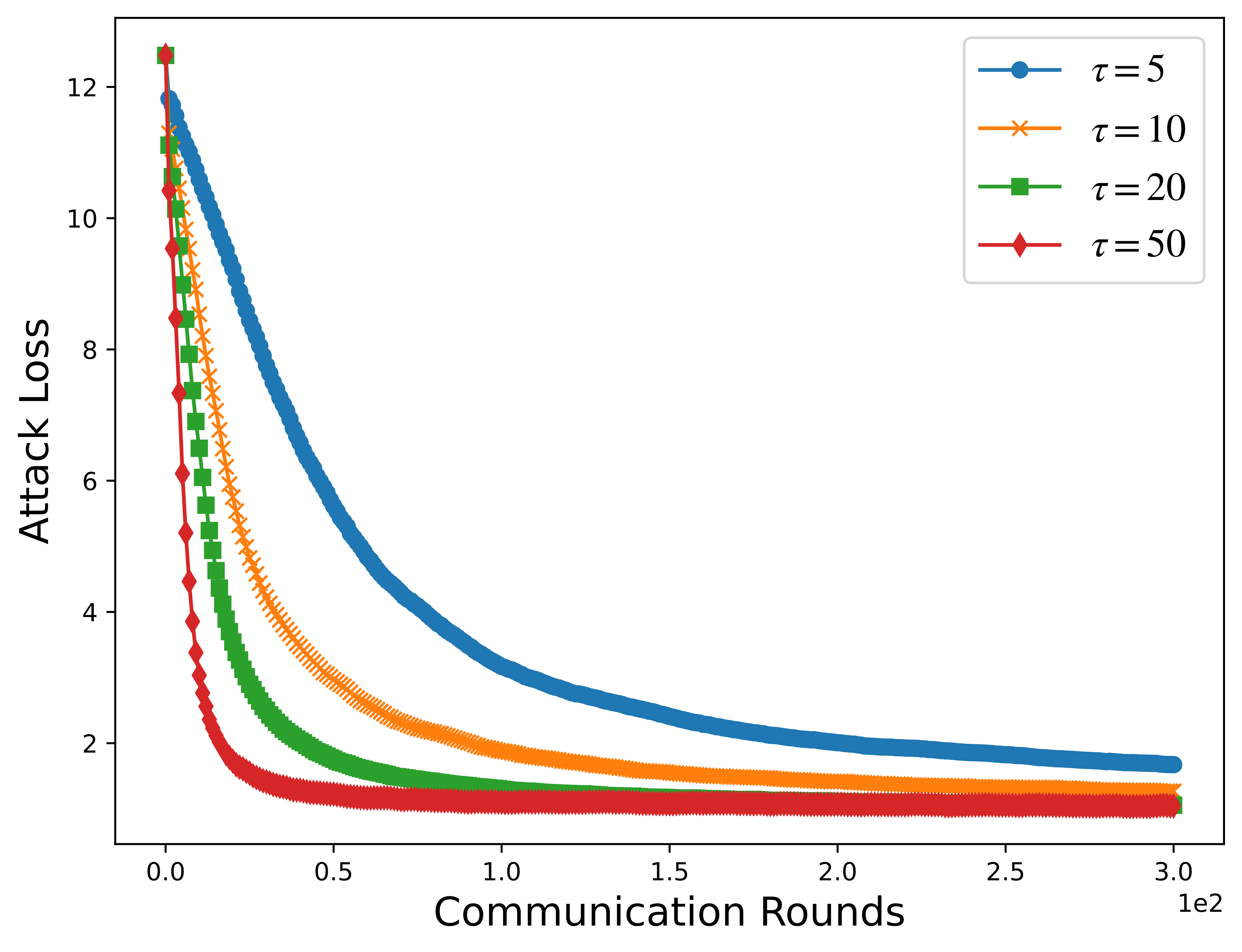}} 
\subfigure[]{\includegraphics[width=0.4\textwidth]{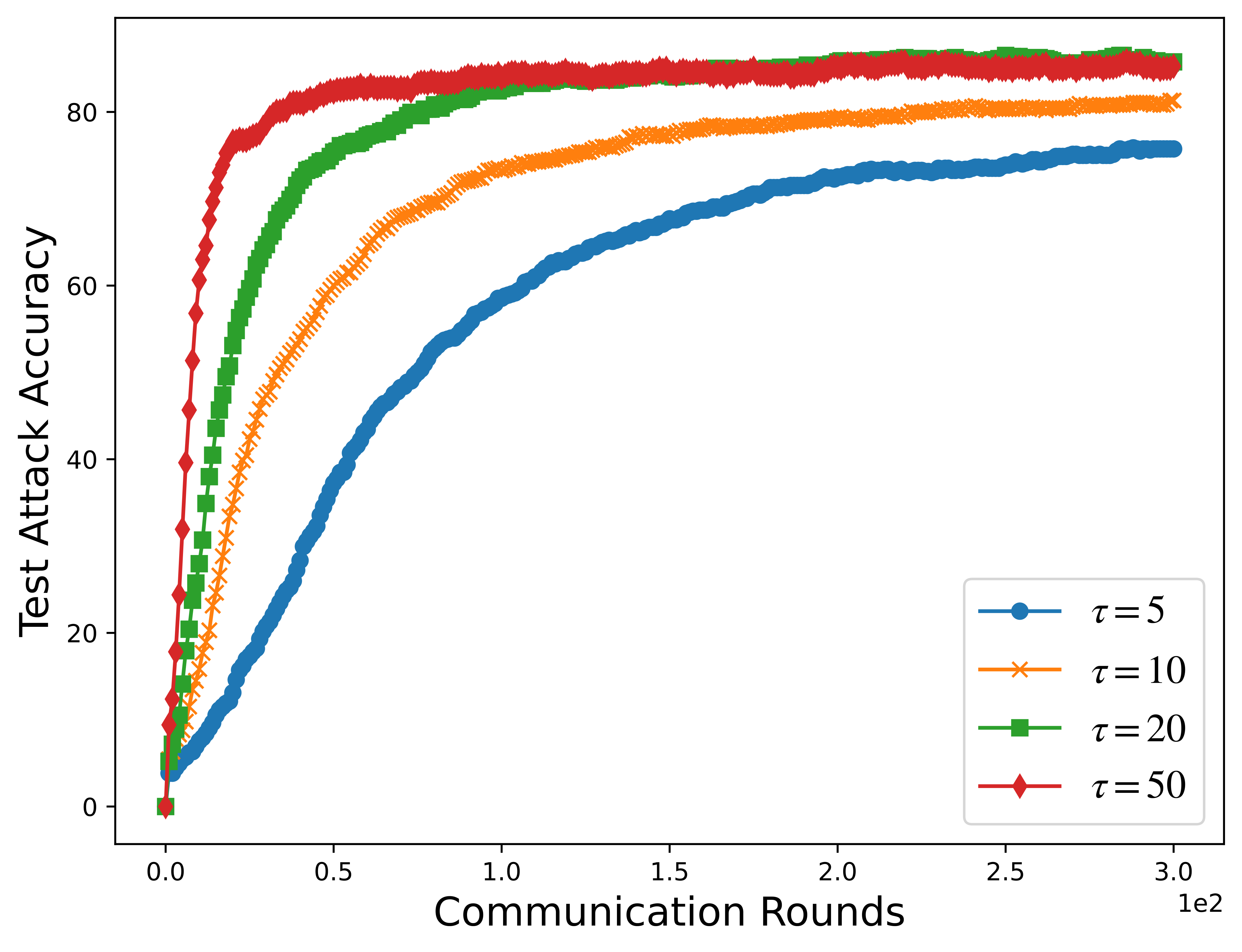}}
\vspace{-10pt}
\caption{Impact of number of local updates.}
\label{fig:blackbox-tau}
\end{figure}
\vspace{-10pt}
\begin{figure}[H]
\centering
\subfigure[]{\includegraphics[width=0.4\textwidth]{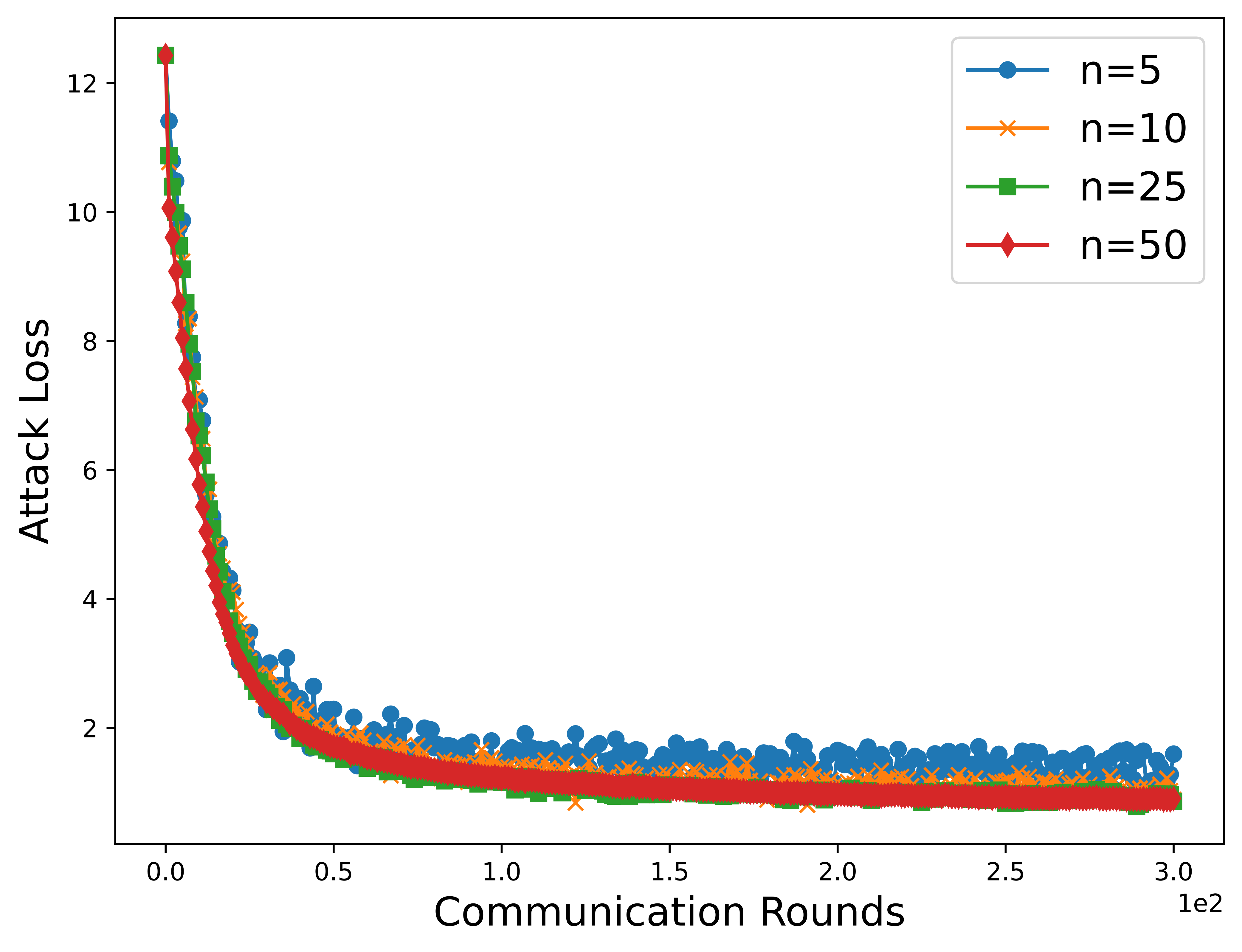}}
\subfigure[]{\includegraphics[width=0.4\textwidth]{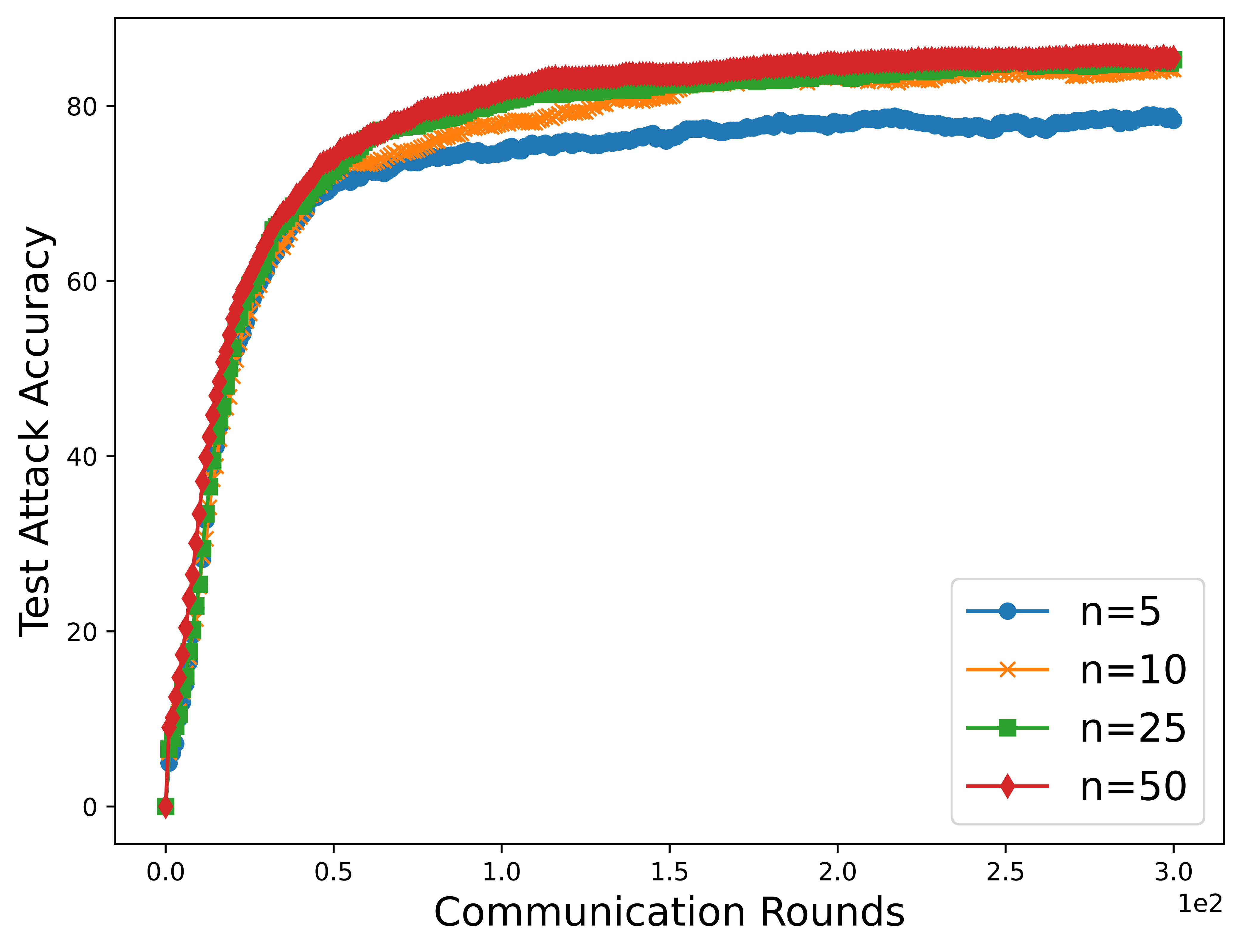}}
\vspace{-10pt}
\caption{Impact of number of participated clients.}
\label{fig:blackbox-n}
\end{figure}
\vspace{-10pt}
\begin{figure}[H]
\centering
\subfigure[]{\includegraphics[width=0.4\textwidth]{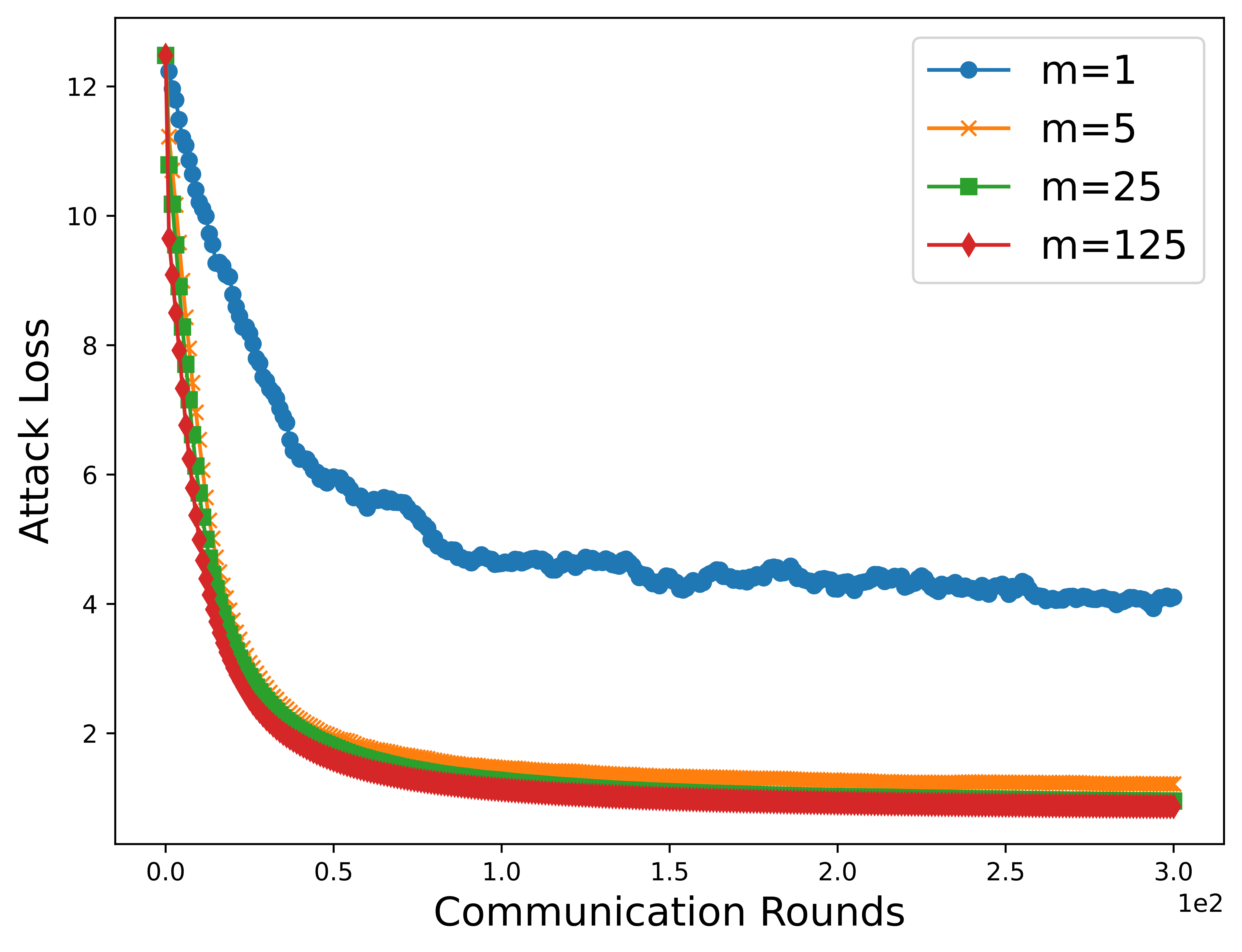}}
\subfigure[]{\includegraphics[width=0.4\textwidth]{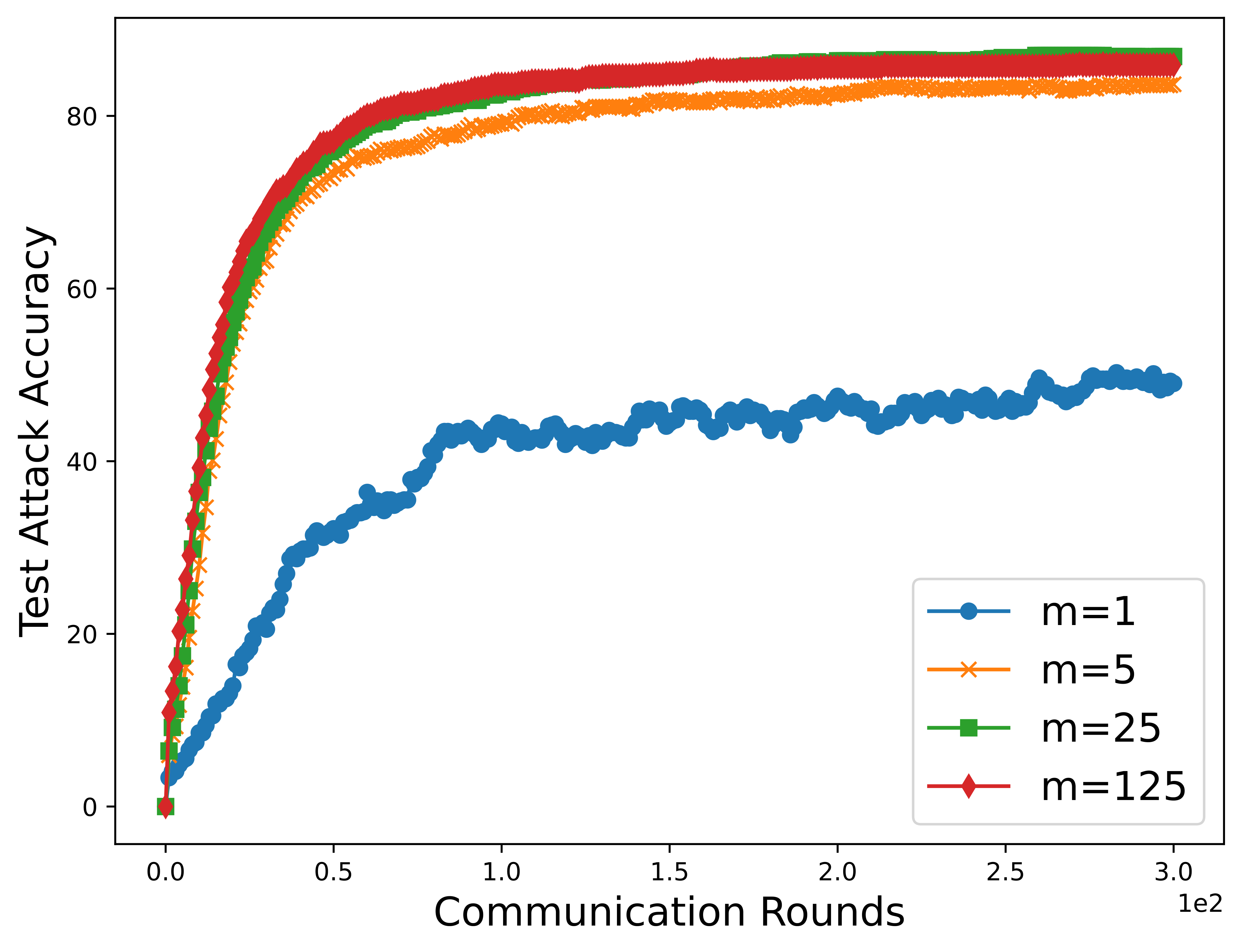}}
\vspace{-10pt}
\caption{Impact of the batch size in zeroth-order Riemannian gradient estimator.}
\label{fig:blackbox-q}
\end{figure}

\subsection{Federated low-rank neural network training}

Low-rank techniques are widely used to reduce computational costs in training machine learning models. By constraining the global parameter space to a low-rank matrix manifold $\cM = \{x \in \mathbb{R}^{p \times r}: \operatorname{rank}(x) = R\}$, the manifold optimization framework ensures an exact low-rank model. We follow the low-rank model compression approach for FL proposed in \citet{xue2023riemannian} to evaluate our algorithm.

We validate our approach on the MNIST dataset \citep{deng2012mnist} using a low-rank multilayer perceptron with rank-constrained hidden layers. As shown in Figure \ref{fig:lowrank-N}, increasing the number of participating clients from 10 to 30 improves test accuracy. Figure \ref{fig:lowrank-q} illustrates the impact of batch size in constructing the zeroth-order gradient estimator: increasing the batch size from 5 to 100 reduces gradient bias. Notably, Figure \ref{fig:lowrank-R} highlights that optimal rank selection strikes a near-optimal balance, reducing communication costs compared to full-rank models while maintaining accuracy. These results demonstrate that our proposed algorithm effectively balances model performance and communication efficiency in low-rank FL training.

\begin{figure}[H]
\centering
\subfigure[]{\includegraphics[width=0.4\textwidth]{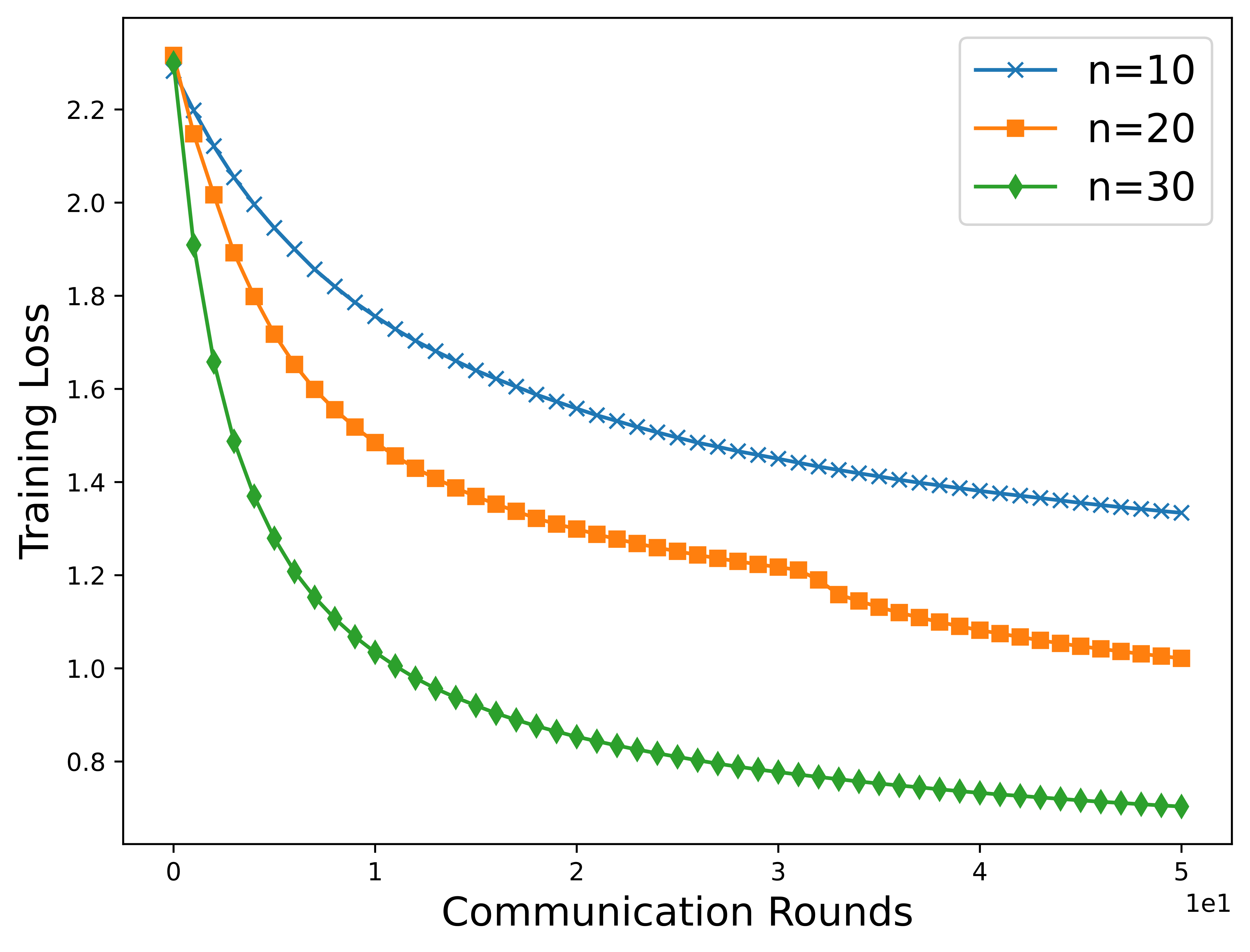}}
\subfigure[]{\includegraphics[width=0.4\textwidth]{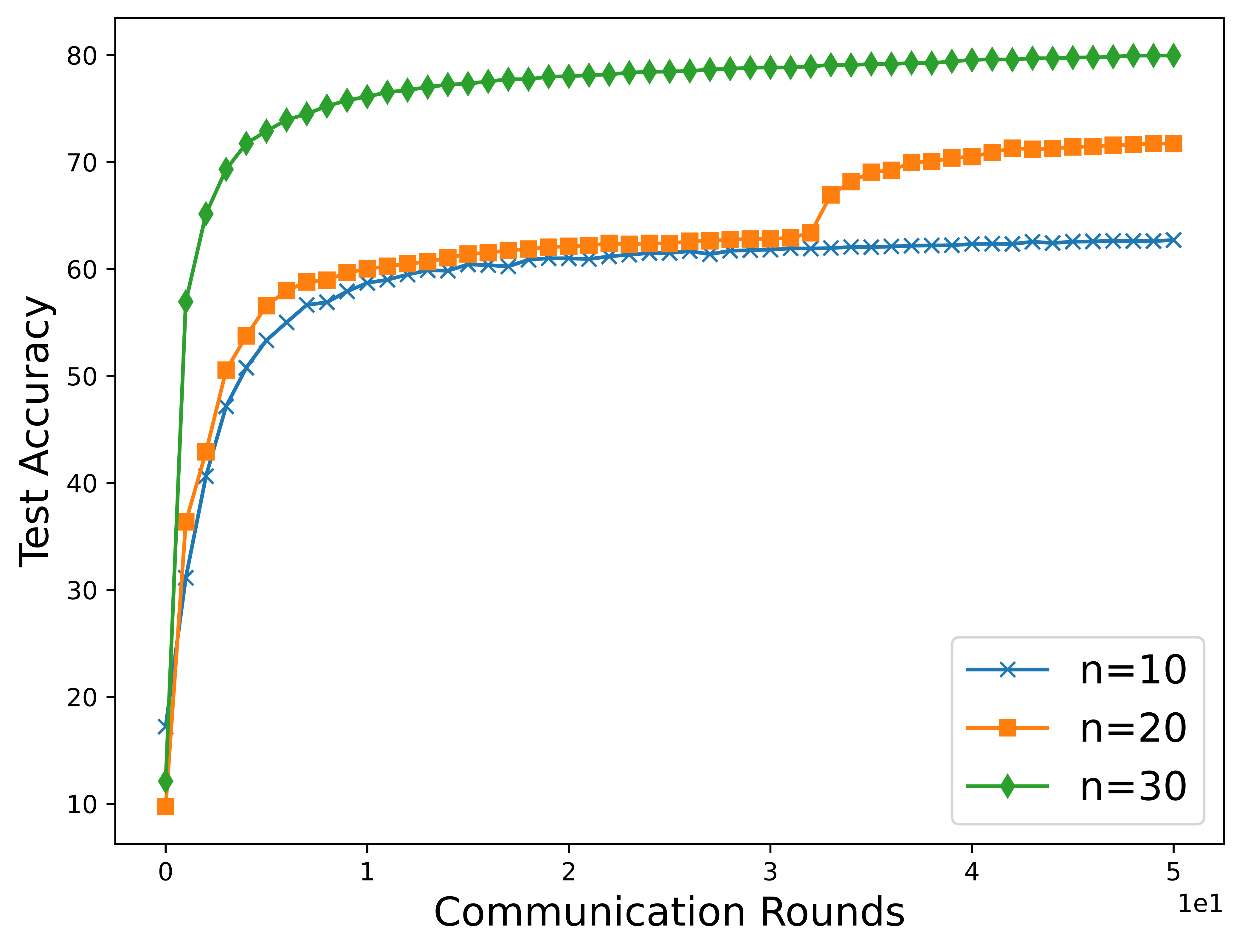}}
\vspace{-10pt}
\caption{Impact of number of participated clients.}
\label{fig:lowrank-N}
\end{figure}
\vspace{-30pt}
\begin{figure}[H]
\centering
\subfigure[]{\includegraphics[width=0.4\textwidth]{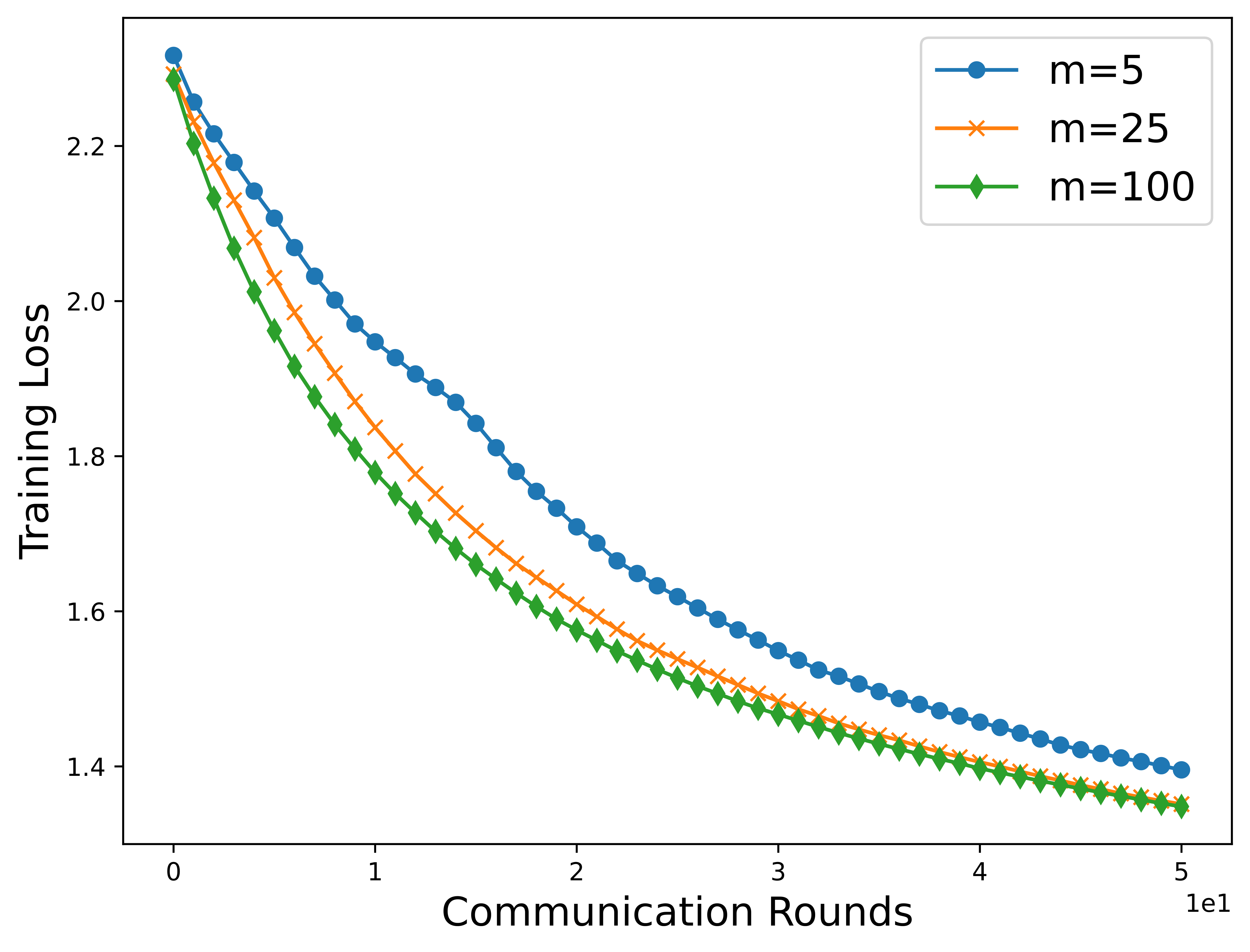}}
\subfigure[]{\includegraphics[width=0.4\textwidth]{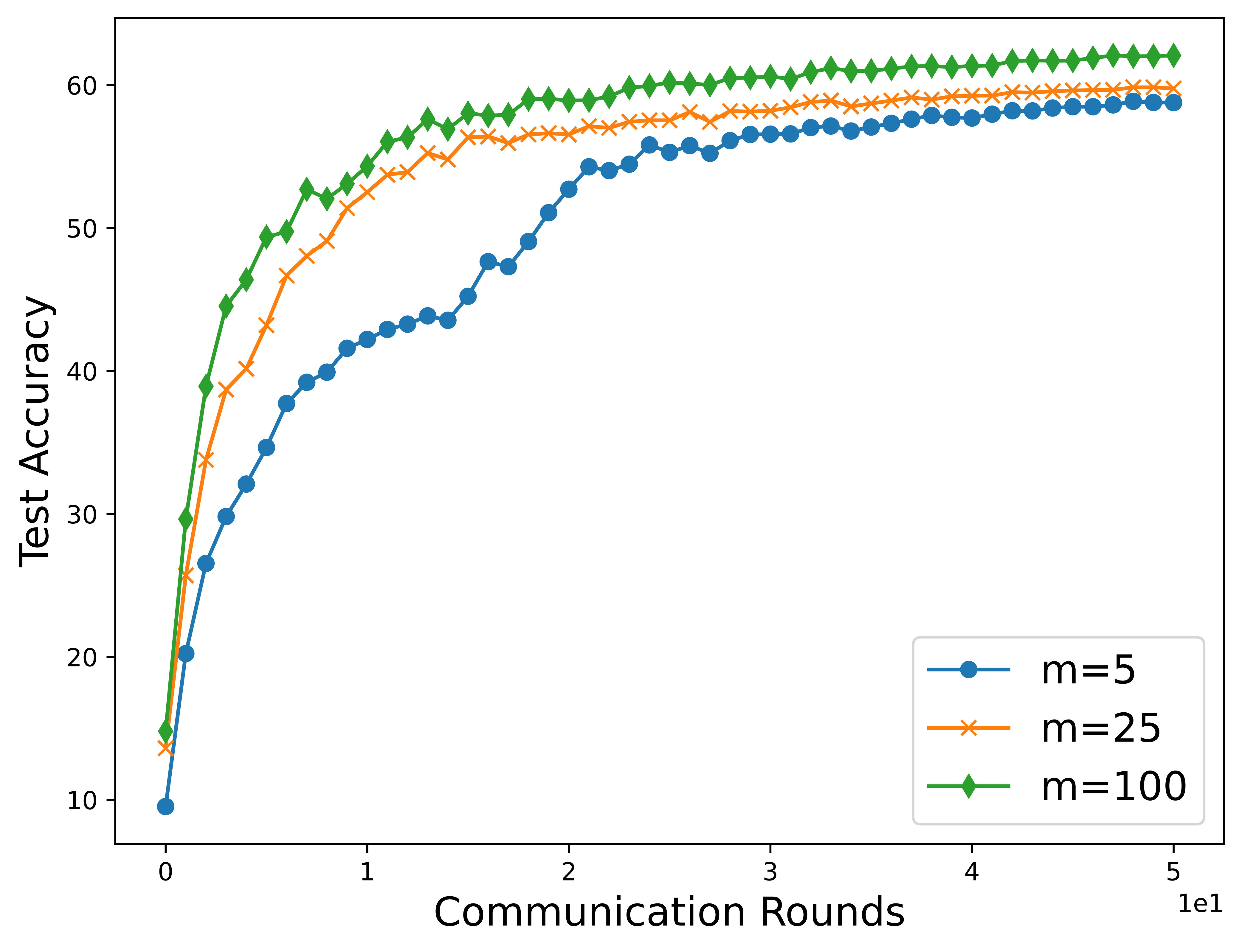}}
\vspace{-10pt}
\caption{Impact of the batch size in zeroth-order Riemannian gradient estimator.}
\label{fig:lowrank-q}
\end{figure}

\begin{figure}[H]
\centering
\subfigure[]{\includegraphics[width=0.4\textwidth]{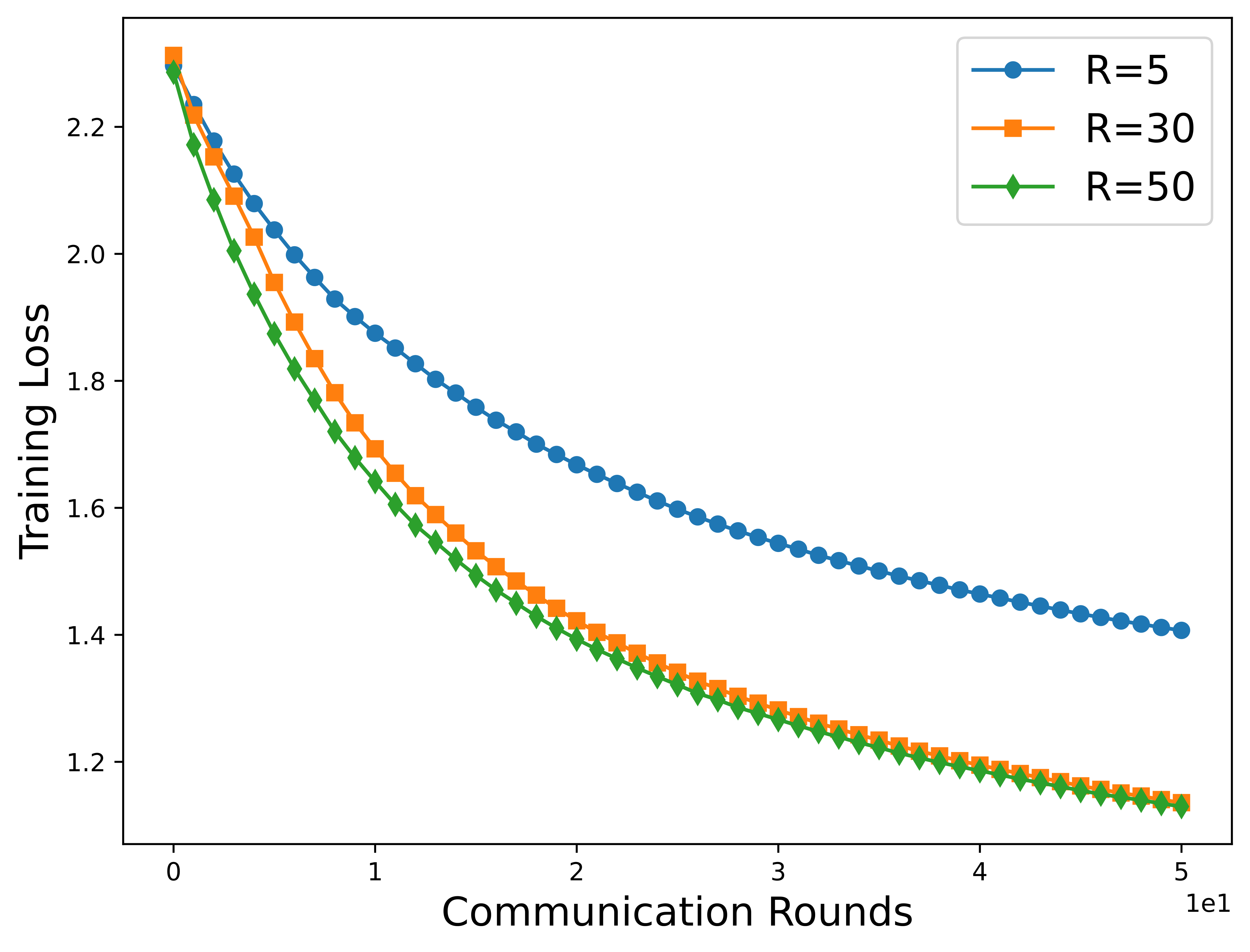}}
\subfigure[]{\includegraphics[width=0.4\textwidth]{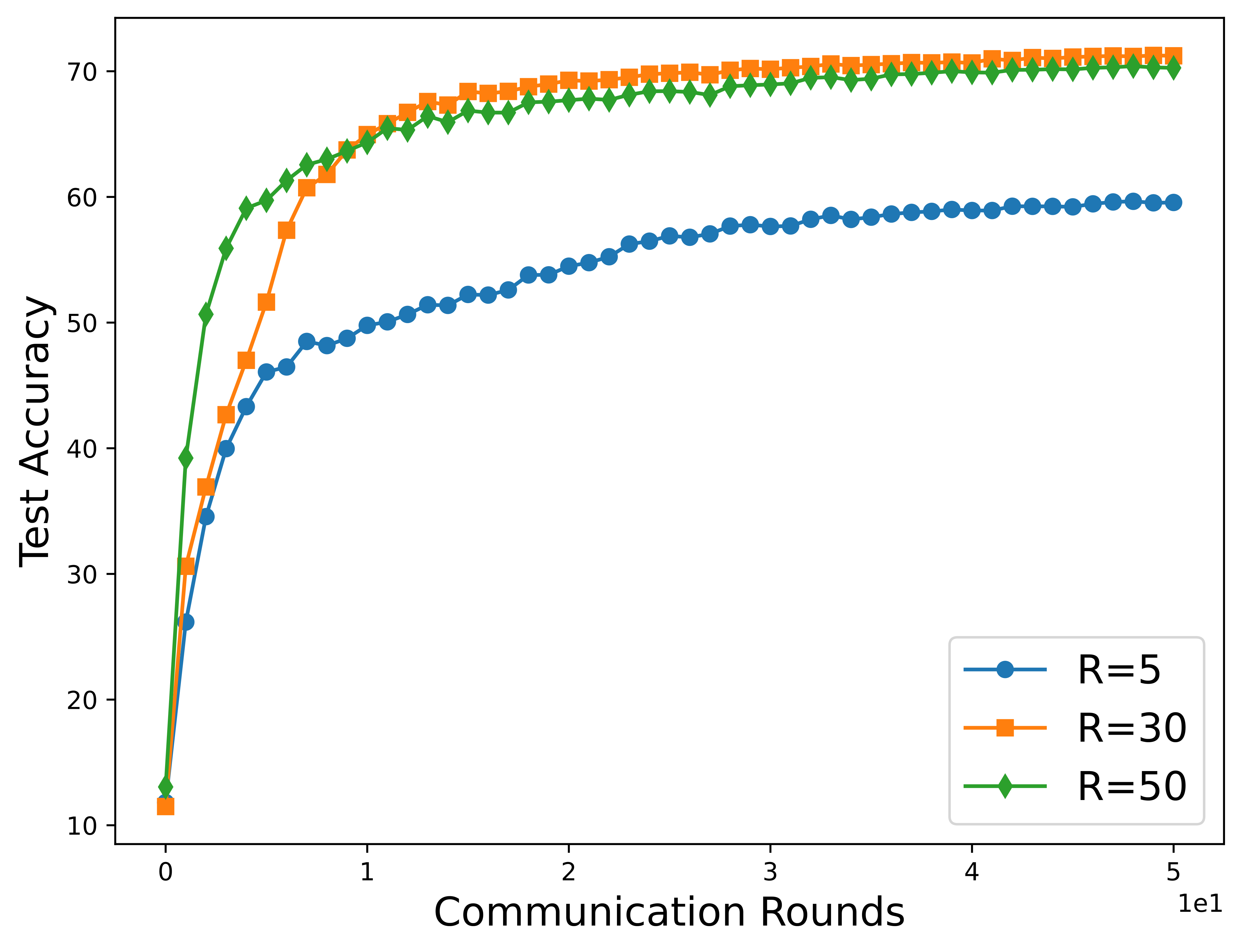}}
\caption{Impact of number of rank.}
\label{fig:lowrank-R}
\end{figure}

\section{Conclusion}
In the paper, we introduce a zeroth-order projection-based Riemannian algorithm for FL, overcoming the limitations of existing Riemannian FL frameworks that rely on exact gradient information. By leveraging the projection operator, our approach only requires simple Euclidean random perturbations, eliminating the need for sampling vectors in the tangent space when constructing the zeroth-order estimator. Additionally, the projection operator simplifies the FL framework by removing the need for sophisticated geometric operations, such as the inverse exponential map and parallel transport. Numerical experiments demonstrate the efficiency of our proposed estimator and support the convergence performance of the Riemannian FL algorithm. Furthermore, it would be interesting to explore how techniques in \citet{xiao2024dissolving} could be applied to eliminate the need for retraction or projection operator in zeroth-order Riemannian optimization.


\clearpage

\clearpage
\bibliographystyle{abbrvnat}
\bibliography{main}

\onecolumn
\appendix
\section*{Appendix}
\section{Technical Lemmas}
\begin{lemma}\label{lemma.cauchy schwarz}
    For any $x, y \in \mathbb{R}^{p \times r}$, and $\varepsilon > 0$, it holds that
    \begin{align*}
        \inner{x}{y} \le \frac{\varepsilon}{2}\|x\|^2 + \frac{1}{2\varepsilon}\|y\|^2.
    \end{align*}
\end{lemma}

\begin{lemma}\label{lemma.projection normal property}
   Suppose Assumption \ref{assum.proximal-smooth} holds. For any $x \in U_{\cM}(2\gamma)$, it holds that $x - \cP_\cM(x) \in \operatorname{N}_{\cP_\cM(x)}(\cM)$. 
\end{lemma}
\begin{proof}
    First notice that the projection is unique when $x \in U_{\cM}(2\gamma)$. From the definition of the projection operator, we have
    \begin{align*}
        0 \ = \ \ &\grad_y \frac{1}{2}\|y -x\|^2 \Big |_{y = \cP_{\cM}(x)} \\
        = \ \ &\cP_{\T_y\cM} \frac{1}{2} \nabla_y \|y - x\|^2 \Big |_{y = \cP_{\cM}(x)} \\
        = \ \ &\cP_{\T_{\cP_{\cM}(x)}\cM} \left(\cP_{\cM}(x) -(x) \right).
    \end{align*}
    Since the normal space $\operatorname{N}_{\cP_\cM(x)}(\cM)$ is orthogonal to the tangent space, and the proof is completed.
\end{proof}

\section{Missing Proofs in Section \ref{section.estimator}}
\paragraph{Proof of Lemma \ref{lemma.lispchitz continuity over Riemannian gradient}:}
\begin{proof}
    Notice that $\grad F(x,\xi)=\nabla F(x,\xi) - \cP_{\operatorname{N}_x\cM}(\nabla F(x,\xi))$, and it follows
    \begin{equation}
        \begin{split}
            &F(y,\xi)-F(x,\xi)-\inner{\grad F(x,\xi)}{y-x} \\
            = \ \ & F(y,\xi)-F(x,\xi)-\inner{\nabla F(x,\xi)}{y-x} + \inner{\cP_{\operatorname{N}_x\cM}(\nabla F(x,\xi))}{y-x}\\
            \le \ \  & \frac{l}{2}\|y-x\|^2+\frac{\|\cP_{\operatorname{N}_x\cM}(\nabla F(x,\xi))\|}{4\gamma}\|y-x\|^2\\
            \le \ \ & \left(\frac{l}{2}+\frac{\chi}{4\gamma}\right)\|y-x\|^2,
        \end{split}
        \nonumber
    \end{equation}
    where the first inequality comes from $l$-smooth property of $F(x,\xi)$ and \eqref{eq.normal}. Then we apply $\|\cP_{\operatorname{N}_x\cM}(\nabla F(x,\xi))\|\le \|\nabla F(x,\xi)\|\le \max_{x\in \cM, \xi\in \cD}\|\nabla F(x,\xi)\|\le \chi$ in the last inequality. Next, we investigate the second inequality. By using the fact that $\grad F(x,\xi)=\cP_{\operatorname{T}_x\cM}(\nabla F(x,\xi))$ and $\grad F(y,\xi)=\cP_{\operatorname{T}_y\cM}(\nabla F(y,\xi))$, we have
    \begin{align*}
        &\|\grad F(y,\xi)-\grad F(x,\xi)\|\\
            = \ \ &\|\cP_{T_y\cM}(\nabla F(y,\xi))-\cP_{\T_x\cM}(\nabla F(x,\xi))\|\\
            \le \ \ & \|\cP_{T_y\cM}(\nabla F(y,\xi))-\cP_{T_y\cM}(\nabla F(x,\xi))\|+\|\cP_{T_y\cM}(\nabla F(x,\xi))-\cP_{\T_x\cM}(\nabla F(x,\xi))\|.
    \end{align*}
    For the first term, it can be easily upper bounded as $\|\cP_{T_y\cM}(\nabla F(y,\xi))-\cP_{T_y\cM}(\nabla F(x,\xi))\| \le \|\nabla F(y,\xi) -\nabla F(x,\xi) \|$ due to the non-expansiveness of projection. Notice that $\cM$ is a smooth compact manifold and $\cP_{\T_x\cM}$ is Lipschitz continuous with respect to $x \in \cM$, then it follows
    \begin{align*}
       \|\cP_{T_y\cM}(\nabla F(x,\xi))-\cP_{\T_x\cM}(\nabla F(x,\xi))\| \le L_{\cP_{T\cM}}\|y-x\| \cdot \|\nabla F(x,\xi)\|, 
    \end{align*}
    where $L_{\cP_{T\cM}}$ is the associate Lipschitz constant. Consequently, we derive
    \begin{align*}
                   &\|\grad F(y,\xi)-\grad F(x,\xi)\|\\
            \le \ \ & \|\nabla F(y,\xi) -\nabla F(x,\xi) \|+\|\nabla F(x,\xi)\|L_{\cP_{T\cM}}\|y-x\|\\
            \le \ \ &l\|y-x\|+\chi L_{\cP_{T\cM}}\|y-x\|\\
            = \ \ &(l+\chi L_{\cP_{T\cM}})\|y-x\|, 
    \end{align*}
    where the second inequality holds due to the $l$-smooth property of $F(x,\xi)$ and $\max_{x\in \cM, \xi\in \cD}\|\nabla F(x,\xi)\|\le \chi$. Let $L= \max\left\{\frac{l}{2}+\frac{\chi}{4\gamma},l+\chi L_{\cP_{T\cM}}\right\}$, and the proof is completed.
\end{proof}
\paragraph{Proof of Lemma \ref{lemma.uniform boundness}:}
\begin{proof}
   First notice that
    \begin{align*}
        &\left|F(\cP_{\cM}(x+\mu \cdot u_{j}), \xi_{j}) - F(x, \xi_{j})\right| \\
        \le \ \  &\left|F(\cP_{\cM}(x+\mu \cdot u_{j}), \xi_{j}) - F(x, \xi_{j}) - \inner{\grad F(x,\xi_{j})}{\cP_{\cM}(x+\mu \cdot u_{j}) - x}\right| + \left|\inner{\grad F(x,\xi_{j})}{\cP_{\cM}(x+\mu \cdot u_{j}) - x}\right| \\
        \le \ \ &\frac{L}{2}\|\cP_{\cM}(x+\mu \cdot u_{j}) - x\|^2 +  \|\grad F(x,\xi_{j})\| \cdot \|\cP_{\cM}(x+\mu \cdot u_{j}) - x\| \\
        \le \ \ &\frac{1}{2}LM^2\mu^2 + 2\chi M\mu ,
    \end{align*}
    where the second inequality comes from \eqref{eq.Lips inequality} and Cauchy inequality, and the last inequality holds due to the inequality \eqref{eq.project Lips} and Assumption \ref{assum.Lipschitz}. Consequently, we have
    \begin{align*}
       \|G^{\mathtt{P}}_{\mu}(x)\| \le \frac{pr}{m\mu} \sum_{j=1}^m \left\|(F(\cP_{\cM}(x+\mu \cdot u_{j}), \xi_{j}) - F(x, \xi_{j}))u_{j}\right\| \le \frac{1}{2}LM^2pr\mu + 2prM\chi,
    \end{align*}
    and the proof is completed.
\end{proof}
\paragraph{Proof of Lemma \ref{lemma.first-order moment}:}
\begin{proof}\label{proof.proof of le 2.6}
    Since $\grad F(x, \xi_j) \in \T_x\cM \subseteq \mathbb{R}^{p \times r}$, Because for any $v \in \R^{p\times r}$, then it holds that $\grad F(x,\xi_{j}) = pr \cdot \E_{u_{j}\sim \mathbb{S}} \left[\inner{\grad F(x,\xi_{j})}{u_{j}} u_{j}\right]$, $j=1,\ldots,m$.
    Also, Assumption \ref{assum.unbiased} implies
    \begin{align*}
       \grad f(x) = \frac{1}{m}\sum_{j=1}^m \E_{\xi_{j}\sim \cD} [\grad F(x,\xi_{j})] = \frac{pr}{m} \sum_{j=1}^m \E [\inner{\grad F(x,\xi_{j})}{u_{j}} u_{j}].
    \end{align*}
    Therefore, we have
    \begin{equation}\label{eq.first-order bound}
        \begin{split}
      &\|\E \left[G^{\mathtt{P}}_{\mu}(x)\right] - \grad f(x)\| \\
      = \ \ &\frac{pr}{\mu m} \cdot \bigg\|\sum_{j=1}^m \E \left[ \left(F(\cP_{\cM}(x + \mu u_{j}), \xi_{j}) - F(x, \xi_{j}) - \inner{\grad F(x, \xi_{j})}{\mu \cdot u_{j}} \right) u_{j}\right]\bigg\| \\
      = \ \ & \frac{pr}{\mu m} \cdot \bigg\|\sum_{j=1}^m \E [ (F(\cP_{\cM}(x + \mu u_{j}), \xi_{j}) - F(x, \xi_{j}) - \inner{\grad F(x, \xi_{j})}{\cP_{\cM}(x+\mu u_j)-x}\\
      &+\inner{\grad F(x, \xi_{j})}{\cP_{\cM}(x+\mu u_j)-(x+\mu u_j)}) u_{j}]\bigg\|\\
     \le & \underbrace{\frac{pr}{\mu m} \cdot \bigg\|\sum_{j=1}^m \E \left[ \left(F(\cP_{\cM}(x + \mu u_{j}), \xi_{j}) - F(x, \xi_{j}) - \inner{\grad F(x, \xi_{j})}{\cP_{\cM}(x+\mu u_j)-x} \right) u_{j}\right]\bigg\|}_{\rm (I)}\\
     &+ \underbrace{\frac{pr}{\mu m} \cdot \bigg\|\sum_{j=1}^m \E \left[ \left(\inner{\grad F(x, \xi_{j})}{\cP_{\cM}(x+\mu u_j)-(x+\mu u_j)}\right) u_{j}\right]\bigg\|}_{\rm (II)}.
     \end{split}
    \end{equation}
    To proceed, we separately bound the two terms above. For the term $\rm (I)$, applying Lipschitz-type inequality \eqref{eq.Lips inequality} gives
    \begin{equation}\label{eq.(I)}
        \begin{split}
            {\rm (I)}\ \le \ \ &\frac{pr}{\mu m} \sum_{j=1}^m \E \left[ \frac{ L}{2}\|\cP_{\cM}(x + \mu u_{j}) - x\|^2 \|u_{j}\|\right] \\
     \le \ \ &\frac{pr}{\mu m} \sum_{j=1}^m \E \left[ \frac{1}{2}LM^2\mu^2\| u_{j}\|^3\right] \\
     = \ \ &\frac{1}{2}LM^2pr\mu,
        \end{split}
    \end{equation}
    where the second inequality comes from \eqref{eq.project Lips}. For the second term $\rm(II)$, notice that $\cP_{\cM}(x+\mu u_j)-(x+\mu u_j) \in \operatorname{N}_{\cP_{\cM}(x+\mu u_j)}\cM$ due to Lemma \ref{lemma.projection normal property}, and then it follows
    \begin{equation}\label{eq.(II)}
        \begin{split}
            {\rm(II)}\ =\ \ &\frac{pr}{\mu m} \cdot \bigg\|\sum_{j=1}^m \E \left[ \left(\inner{\grad F(x, \xi_{j})-\grad F(\cP_{\cM}(x+\mu u_j), \xi_{j})}{\cP_{\cM}(x+\mu u_j)-(x+\mu u_j)}\right) u_{j}\right]\bigg\|\\
            \le \ \ &\frac{pr}{\mu m}\sum_{j=1}^m\E\left[L\|\grad F(x, \xi_{j})-\grad F(\cP_{\cM}(x+\mu u_j), \xi_{j})\| \cdot \|\cP_{\cM}(x+\mu u_j)-(x+\mu u_j)\| \cdot \|u_j\|\right]\\
            \le \ \ &\frac{pr}{\mu m}\sum_{j=1}^m\E\left[L\|\cP_{\cM}(x+\mu u_j)-x\| \cdot \|\cP_{\cM}(x+\mu u_j)-(x+\mu u_j)\| \cdot \|u_j\|\right]\\
            \le \ \ &\frac{pr}{\mu m}\sum_{j=1}^m\E\left[LM\mu \|x-(x+\mu u_j)\| \cdot \|u_j\|^2\right]\\
            = \ \ &\frac{LMpr\mu}{m}\sum_{j=1}^m\E[\|u_j\|^3]=LMpr\mu,
        \end{split}
    \end{equation}
    where the second inequality holds due to the Lipschitz property in Lemma \ref{lemma.lispchitz continuity over Riemannian gradient}. Then, from the definition of the projection operator, it follows that $\|\cP_{\cM}(x+\mu u_j)-(x+\mu u_j)\|\le \|x-(x+\mu u_j)\|$. Combining this with the projection property \eqref{eq.project Lips} establishes the third inequality. Finally, by substituting the bounds \eqref{eq.(I)} and \eqref{eq.(II)} into equation \eqref{eq.first-order bound}, and setting $\chi_f=\frac{1}{2}LM^2pr+LMpr$, we complete the proof.
\end{proof}
\paragraph{Proof of Lemma \ref{lemma.second-order moment}:}
\begin{proof}
    First notice that
    \begin{align*}
        &\E \left[\left\|G^{\mathtt{P}}_{\mu}(x) - \grad f(x)\right\|^2\right] \\
        \le \ \ &2 \cdot \E \left[\left\|\grad f(x) - \frac{1}{m}\sum_{j=1}^m \grad F(x, \xi_{j})\right\|^2\right] + 2 \cdot \E \left[\left\|G^{\mathtt{P}}_{\mu}(x) - \frac{1}{m}\sum_{j=1}^m\grad F(x, \xi_{j})\right\|^2\right].
    \end{align*}
    For the first term, it can be bounded as
    \begin{align*}
       \E \left[\left\|\grad f(x) - \frac{1}{m}\sum_{j=1}^m \grad F(x, \xi_{j})\right\|^2\right] \le \frac{\sigma^2}{m}
    \end{align*}
    due to Assumption \ref{assum.unbiased} and the batching property. For the second term, we first inspect the expectation for $U=\left\{u_{1}, \ldots, u_{m}\right\}$:
    \begin{align*}
        &\E_{U} \left[\left\|G^{\mathtt{P}}_{\mu}(x) - \frac{1}{m}\sum_{j=1}^m\grad F(x, \xi_{j})\right\|^2\right] \\
        = \ \ &\frac{1}{m^2} \cdot \E_{U} \left[\left\|\sum_{j=1}^m \frac{pr}{\mu}\left(F(\cP_{\cM}(x + \mu u_{j}), \xi_{j}) - F(x, \xi_{j})\right)u_{j} - \grad F(x, \xi_{j})\right\|^2\right].
    \end{align*}
    For brevity, we introduce the notation
    \begin{equation*}
        g_{j} :=  \frac{pr}{\mu}\left(F(\cP_{\cM}(x + \mu u_{j}), \xi_{j}) - F(x, \xi_{j})\right)u_{j} - \grad F(x, \xi_{j}),
    \end{equation*}
    and then proceed as follows:
    \begin{align*}
        &\E_{U} \left[\left\|G^{\mathtt{P}}_{\mu}(x) - \frac{1}{m}\sum_{j=1}^m\grad F(x, \xi_{j})\right\|^2\right] \\
        = \ \ & \frac{1}{m^2} \cdot \E_{U} \left[\sum_{j=1}^m\left\|g_{j}\right\|^2\right] + \frac{1}{m^2} \cdot \E_{U} \left[\sum_{1 \le j,l \le m, j \neq l}\inner{g_{j}}{g_{l}}\right] \\
        = \ \ & \frac{1}{m} \cdot \E_{u_{1}} \left[\left\|g_{1}\right\|^2\right] + \frac{1}{m^2} \cdot \sum_{1 \le j,l \le m, j \neq l}\inner{\E_{u_{j}} [g_{j}]}{\E_{u_{l}} [g_{l}]},
    \end{align*}
    where the last inequality follows from that $u_{1},\ldots,u_{m}$ are i.i.d.. Notice that
    \begin{align*}
        \E_{u_{1}} \left[\left\|g_{1}\right\|^2\right] \le \frac{2p^2r^2}{\mu^2} \cdot \E_{u_{1}} \left[\left(F(\cP_{\cM}(x+\mu u_{1}), \xi_{1}) - F(x, \xi_{1})\right)^2\|u_{1}\|^2\right] + 2 \cdot \|\grad F(x,\xi_{1})\|^2,
    \end{align*}
    and 
    \begin{align*}
        &\left(F(\cP_{\cM}(x+\mu u_{1}), \xi_{1}) - F(x, \xi_{1})\right)^2 \\
        = \ \ &\left(F(\cP_{\cM}(x+\mu u_{1}), \xi_{1}) - F(x, \xi_{1}) - \inner{\grad F(x,\xi_{1})}{\cP_{\cM}(x+\mu u_{1}) - x} + \inner{\grad F(x,\xi_{1})}{\cP_{\cM}(x+\mu u_{1}) - x}\right)^2 \\
        \le \ \ &2 \cdot \left(F(\cP_{\cM}(x+\mu u_{1}), \xi_{1}) - F(x, \xi_{1}) - \inner{\grad F(x,\xi_{1})}{\cP_{\cM}(x+\mu u_{1}) - x}\right)^2 + 2\inner{\grad F(x,\xi_{1})}{\cP_{\cM}(x+\mu u_{1}) - x}^2\\
        \le \ \ &2 \cdot \left(\frac{1}{2}LM^2\mu^2\|u_{1}\|^2\right)^2 + 2(M\|\mu u_{1}\|)^2\|\grad F(x,\xi_{1})\|^2 \\
        = \ \ &\frac{1}{2}L^2M^4\mu^4 + 2M^2\mu^2\|\grad F(x,\xi_{1})\|^2.
    \end{align*}
    We use Lemma \ref{lemma.lispchitz continuity over Riemannian gradient} and projection property \eqref{eq.project Lips} in the above second inequality. Consequently, we obtain
    \begin{align*}
      \E_{u_{1}} \left[\left\|g_{1}\right\|^2\right] \le  L^2M^4p^2r^2\mu^2 + (4M^2p^2r^2+2)\|\grad F(x,\xi_{1})\|^2.
    \end{align*}
    Moreover, since $\grad F(x,\xi_{j}) = pr \cdot \E_{u_{j}} \left[\inner{\grad F(x,\xi_{j})}{u_{j}} u_{j}\right]$ holds for any $j$, and it follows
    \begin{align*}
    &\E_{u_{j}} [g_{j}] \\
    = \ \ &\E_{u_{j}} \left[\frac{pr}{\mu} (F(\cP_{\cM}(x + \mu u_{j}), \xi_{j}) - F(x, \xi_{j})) u_{j}\right] - \grad F(x,\xi_{j}) \\
    = \ \ &\frac{pr}{\mu} \cdot \E_{u_{j}} \left[\left(F(\cP_{\cM}(x + \mu u_{j}), \xi_{j}) - F(x, \xi_{j}) - \inner{\grad F(x,\xi_{j})}{\mu \cdot u_{j}}\right)u_{j}\right].   
    \end{align*}
    Applying a similar argument in \eqref{eq.first-order bound} gives $\left\|\E_{u_{j}} [g_{j}]\right\| \le (\frac{1}{2}LM^2pr+LMpr)\mu$, and hence we derive
    \begin{align*}
        \sum_{1 \le j,l \le m, j \neq l}\inner{\E_{u_{j}} [g_{j}]}{\E_{u_{l}} [g_{l}]} \le m(m-1) \cdot (\frac{1}{2}LM^2pr+LMpr)^2\mu^2.
    \end{align*}
    To this end, we have
    \begin{align*}
        &\E_{U} \left[\left\|G^{\mathtt{P}}_{\mu}(x) - \frac{1}{m}\sum_{j=1}^m\grad F(x, \xi_{j})\right\|^2\right] \\
        \le \ \ &\frac{m-1}{m}(\frac{1}{2}LM^2pr+LMpr)^2\mu^2 + \frac{1}{m}L^2M^4p^2r^2\mu^2 + \frac{4M^2p^2r^2+2}{m}\|\grad F(x,\xi_{1})\|^2.
    \end{align*}
    Now take expectation for $\Xi=\left\{\xi_{1}, \ldots, \xi_{m}\right\}$. Assumption \ref{assum.Lipschitz} implies
    \begin{align*}
       &\E \left[\left\|G^{\mathtt{P}}_{\mu}(x) - \frac{1}{m}\sum_{j=1}^m\grad F(x, \xi_{j})\right\|^2\right] \\
       \le  \ \ &\frac{m-1}{m}(\frac{1}{2}LM^2pr+LMpr)^2\mu^2 + \frac{1}{m}L^2M^4p^2r^2\mu^2 + \frac{4M^2p^2r^2+2}{m}\chi^2.
    \end{align*}
    By putting all things together, we conclude
    \begin{align*}
        &\E \left[\left\|G^{\mathtt{P}}_{\mu}(x) - \grad f(x)\right\|^2\right]\\
        \le& \frac{2m-2}{m}(\frac{1}{2}LM^2pr+LMpr)^2\mu^2 +\frac{2L^2M^4p^2r^2\mu^2+(8M^2p^2r^2+4)\chi^2+2\sigma^2}{m}.
    \end{align*}
    Let $\chi_1=\frac{2m-2}{m}(\frac{1}{2}LM^2pr+LMpr)^2$, $\chi_2=2L^2M^4p^2r^2\mu^2+(8M^2p^2r^2+4)\chi^2+2\sigma^2$, the result is proved.
\end{proof}

\section{Missing Proofs in Section \ref{section.algorithm}}
\paragraph{Proof of Lemma \ref{lemma.special}:}
\begin{proof}
    Denote $h(y)=\frac{1}{2\eta}\|y-(x-\eta v)\|^2$, and it is easy to see that $h(y)$ is a $\frac{1}{\eta}$-strongly convex function. For any $z\in \cM$, it implies 
    \begin{equation}\label{eq.1/eta convex}
        \frac{1}{2\eta}\|z-(x-\eta v)\|^2\ge \frac{1}{2\eta}\|x^+-(x-\eta v)\|^2 + \inner{\frac{1}{\eta}\left(x^+-(x-\eta v)\right)}{z-x^+} + \frac{1}{2\eta}\|z-x^+\|^2.
    \end{equation}
    For the inner product term in the above inequality, since $x^+$ is the solution of $\min_{y\in \cM}h(y)$ and $\nabla h(y) = (y - (x - \eta v)) / \eta$, we have
    \begin{equation}
        \begin{split}
            \inner{\frac{1}{\eta}\left(x^+-(x-\eta v)\right)}{z-x^+} \ = \ \ & \inner{\frac{1}{\eta}\left(x^+-(x-\eta v)\right)-\grad h(x^+)}{z-x^+}\\
            \ge \ \ & -\frac{\|\frac{1}{\eta}\left(x^+-(x-\eta v)\right)-\grad h(x^+)\|}{4\gamma}\|z-x^+\|^2\\
            \ge \ \ & -\frac{\|\frac{1}{\eta}\left(x^+-(x-\eta v)\right)\|}{4\gamma}\|z-x^+\|^2\\
            \ge \ \ & -\frac{3\|v\|}{4\gamma}\|z-x^+\|^2.
        \end{split}\notag
    \end{equation}
    The first inequality comes from property \eqref{eq.normal}, and the second inequality is due to $\|\nabla h(y)-\grad h(y)\|\le \|\nabla h(y)\|$. Notice that $x-\eta v\in \overline{U}_{\cM}(\gamma)$, and we may apply property \eqref{eq.lips=2} to get $\|x^+-(x-\eta v)\|=\|\proj{x-\eta v}-(x-\eta v)\|\le \|\proj{x-\eta v}-x\| +\eta\|v\| = \|\proj{x-\eta v}-\proj{x}\| +\eta\|v\| \le 3\eta\|v\|$. Then the last inequality in above equation is proved. Combining the above inequality with \eqref{eq.1/eta convex} leads to
    $$
        \frac{1}{2\eta}\|z-(x-\eta v)\|^2\ge \frac{1}{2\eta}\|x^+-(x-\eta v)\|^2 + \left(\frac{1}{2\eta}-\frac{3\|v\|}{4\gamma}\right)\|z-x^+\|^2
    $$
    By expanding the quadratic terms $\|z-(x-\eta v)\|^2$ and $\|x^+-(x-\eta v)\|^2$ and rearranging the above inequality, it follows
    \begin{equation}\label{eq.bound inner v and z-x+}
        \inner{v}{z-x^+}\ge \frac{1}{2\eta}(\|x^+-x\|^2-\|z-x\|^2)+\left(\frac{1}{2\eta}-\frac{3\|v\|}{4\gamma}\right)\|z-x^+\|^2.
    \end{equation}
    Now using Lipschitz-type inequality \eqref{eq.Lips inequality} implies
    \begin{align*}
        f(x^+) \ \le \ \ &f(x)+\inner{\grad f(x)}{x^+-x}+\frac{L}{2}\|x^+-x\|^2\\
        \le \ \  & f(z)+\inner{\grad f(x)}{x-z} + \inner{\grad f(x)}{x^+-z} + \frac{L}{2}\|z-x\|^2+\frac{L}{2}\|x^+-x\|^2,
        \\
        = \ \  & f(z)+\inner{\grad f(x)}{x^+-z}+\frac{L}{2}\|z-x\|^2+\frac{L}{2}\|x^+-x\|^2,
    \end{align*}
   It remains to substitute equation \eqref{eq.bound inner v and z-x+} and the proof is completed.
\end{proof}

Before delving into the proof of Theorem \ref{theorem.convergence}, we introduce some notations to make the analysis more clarity. For matrices $z_1,\cdots,z_n\in\R^{p\times r}$, we define $\bz \triangleq \col{z_i}_{i = 1}^n = [z_1;\cdots;z_n]\in\R^{np\times r}$ as the vertical stack of all the matrices and further, we let $\mathbf{x}\triangleq\col{x}_{i=1}^n=[x;\cdots;x]\in \R^{np\times r}$. The superscript $k$ is the index of the communication round, and $t$ is the index of the local update. The subscript $i$ denotes the $i$-th clients. Given local variable $z_i^{k,t}$, we denote the stack of gradient estimator of all clients as $\bG^{k,t} \triangleq \col{G_{i}^{k,t}}_{i=1}^n$, and the stack of average gradient estimator $\overline{\bG}^{k,t}\triangleq \col{\frac{1}{n}\sum_{i=1}^nG_{i}^{k,t}}_{i=1}^n$. Similarly, we define the stack of Riemannian gradient $\grad \f(\bz^{k,t}) \triangleq \col{\grad f_i(z_i^{k,t})}_{i=1}^n$ and the stack of average Riemannian gradient $\overline{\grad \f(\bz^{k,t})} \triangleq \col{\frac{1}{n}\sum_{i=1}^n\grad f_i(z_{i}^{k,t})}_{i=1}^n$. Moreover, if we have $\bz = \col{z_i}_{i=1}^n$, then $\cP_{\cM}(\bz) = \col{\cP_{\cM}(z_i)}_{i=1}^n$. We introduce the following potential function $\Omega^k$, defined as:
\begin{equation}\label{eq.omega^k}
    \Omega^k \triangleq f(\proj{x^k})-f^*+\frac{1}{n\tilde{\eta}}\|\blambda^k - \overline{\blambda}^k\|^2,
\end{equation}
where $f^*$ is the optimal value of problem (\ref{eq.main}) and
\begin{subequations}
    \begin{align}
       \blambda^k &\triangleq \eta\left( \tau\grad\f(\proj{\bx^k}) + \sum_{t = 0}^{\tau-1}\overline{\bG}^{k-1,t}-\sum_{t = 0}^{\tau-1}\bG^{k-1,t}\right), \label{eq.Lambda^k} \\
        \overline{\blambda}^k &\triangleq \col{\frac{1}{n}\sum_{i=1}^{n}\Lambda_i^k}_{i=1}^n. \label{eq.ave of Lambda^k}
    \end{align}
\end{subequations}
All the notations is summarized in the Table \ref{table.notations}. Finally, let us recall the optimality measure $\cG_{\tilde{\eta}}(\proj{x^k}) \triangleq (\proj{x^k} - \tilde{x}^{k+1})/\tilde{\eta}$, where 
\begin{equation}\label{eq.PRGD}
    \tilde{x}^{k+1}\triangleq\proj{\proj{x^k}-\tilde{\eta}\grad f(\proj{x^k})}.
\end{equation} 
\begin{table}[H]
    \centering
    \resizebox{\textwidth}{!}{%
    \begin{tabular}{|c|c|c|c|}
        \hline
        Notation & Specific Formulation & Specific Meaning & Dimension\\
        \hline
        $\bz$ & $[z_1;\cdots;z_n]$ & Vertical stack of all $z_i$ & $\R^{np\times r}$ \\
        \hline
        $\bx$ & $[x;\cdots;x]$ & Vertical stack of $x$ & $\R^{np\times r}$ \\
        \hline
        $\bG^{k,t}$ & $[G^{k,t}_i;\cdots;G^{k,t}_n]$ & Vertical stack of all $G^{k,t}_i$ & $\R^{np\times r}$\\
        \hline
        $\overline{\bG}^{k,t}$ & $[\frac{1}{n}\sum_{i=1}^nG_{i}^{k,t};\cdots;\frac{1}{n}\sum_{i=1}^nG_{i}^{k,t}]$ & Vertical stack of $\frac{1}{n}\sum_{i=1}^nG_{i}^{k,t}$ & $\R^{np\times r}$\\
        \hline
        $\grad \f(\bz^{k,t})$ & $[\grad f_i(z_1^{k,t});\cdots;\grad f_i(z_n^{k,t})]$ & Vertical stack of all $\grad f_i(z_i^{k,t})$ & $\R^{np\times r}$\\
        \hline
        $\overline{\grad \f(\bz^{k,t})}$ & $[\frac{1}{n}\sum_{i=1}^n\grad f_i(z_{1}^{k,t});\cdots;\frac{1}{n}\sum_{i=1}^n\grad f_i(z_{n}^{k,t})]$ & Vertical stack of $\frac{1}{n}\sum_{i=1}^n\grad f_i(z_{i}^{k,t})$ & $\R^{np\times r}$\\
        \hline
        $\cP_{\cM}(\bz)$ & $[\cP_{\cM}(z_1);\cdots;\cP_{\cM}(z_n)]$ & Vertical stack of all $\cP_{\cM}(z_i)$ & $\R^{np\times r}$\\
        \hline
        $\blambda^k$ & \eqref{eq.Lambda^k} & Auxiliary function & $\R^{np\times r}$\\
        \hline
        $\overline{\blambda}^k$ & \eqref{eq.ave of Lambda^k} & Auxiliary function & $\R^{np\times r}$\\
        \hline
        $\Omega^k$ & 
        \eqref{eq.omega^k} & Potential function & $\R^{np\times r}$\\
        \hline
    \end{tabular}
    }
    \caption{Notations for analysis of Algorithm \ref{alg.federated}}
    \label{table.notations}
\end{table}
 \begin{lemma}\label{lemma.equivalent update}
        The update of Algorithm \ref{alg.federated} is equivalent to
        \begin{equation}
            \begin{cases}
                \hat{\bz}^{k,t+1} = \hat{\bz}^{k,t} -\eta\left(\bG^{k,t} + \frac{1}{\tau}\sum_{t=1}^{\tau-1}\overline{\bG}^{k-1,t}-\frac{1}{\tau}\sum_{t=1}^{\tau-1}\bG^{k-1,t}\right),\\
                \bz^{k,t+1} = \proj{\hat{\bz}^{k,t+1}},\\
                \bx^{k+1} = \proj{\bx^k}-\eta_g\eta\sum_{t=0}^{\tau-1}\overline{\bG}^{k,t}.
            \end{cases}
        \end{equation}
        Moreover, if we let $G_{i}^{0,t}=0$, for all $t = 0,\cdots,\tau-1$ and $i=1,\cdots,n$, we have $c_i^k = \frac{1}{\tau}\sum_{t=0}^{\tau -1}\frac{1}{n}\sum_{i=1}^nG_{i}^{k-1,t}-\frac{1}{\tau}\sum_{t=0}^{\tau -1} G_{i}^{k-1,t}$.
    \end{lemma}
    \begin{proof}
        This lemma can be proved by a similar argument in \cite{zhang2024composite}, and we omit it here.
    \end{proof}
    
    \begin{lemma}
         Suppose Assumptions \ref{assum.proximal-smooth}, \ref{assum.unbiased} and \ref{assum.Lipschitz} hold. If the step size $\tilde{\eta}\triangleq \eta_g\eta\tau\le \min\left\{\frac{\eta_g\gamma}{\chi},\frac{\eta_g\gamma}{3\chi_G},\frac{\eta_g}{16L}\right\}$, then we have
         \begin{equation}\label{eq.second term bound}
            \begin{split}
                \frac{1}{n}\E\|\blambda^{k+1}-\overline{\blambda}^{k+1}\|
                \ \le \ \ & 2\eta^2L^2\tau^2\E\|\proj{\bx^{k+1}}-\proj{\bx^{k}}\|^2\\
                 &+\frac{4\eta^2L^2\tau}{n}\Bigg(3nM^2\tau^3\eta^2\|\grad f(\proj{x^k})\|^2+9\tau\E\left\|\blambda_i^k-\overline{\blambda}^k\right\|^2\\
                &+ 18n\eta^2\tau^2(\chi_1\mu^2+\chi_2/m)+18n\eta^2\tau^2(\tau-1)\left(\chi_G + \chi\right)\chi_f\mu\Bigg)\\
                &+4\eta^2\tau(\chi_1\mu^2+\chi_2/m)+4\eta^2\tau(\tau-1)\left(\chi_G + \chi\right)\chi_f\mu.
            \end{split}
         \end{equation}
    \end{lemma}
    \begin{proof}
        Firstly, we bound the drift error $\|z_{i}^{k,t+1}-\proj{x^k}\|$. When $\tau = 1$, the error is $0$ due to $z_{i}^{k,t}=\proj{x^k}$. When $\tau\ge 2$, from the update of the Algorithm \ref{alg.federated}, we have 
        \begin{equation}\label{eq.z-p}
            \E\|z_{i}^{k,t+1}-\proj{x^k}\|^2=\E\left\|\proj{\proj{x^k}-\eta \sum_{l=0}^t(G_{i}^{k,l} +c_i^k)}-\proj{x^k}\right\|^2.
        \end{equation}
        Take $\tilde{\eta}=(t+1)\eta$ in equation (\ref{eq.PRGD}), i.e., $\tilde{x}^{k+1}=\proj{\proj{x^k}-(t+1)\eta\grad f(\proj{x^k})}$. We bound the equation (\ref{eq.z-p}) by compare with the above equation:
        \begin{equation}\label{eq.(z-p)1}
            \begin{split}
                &\E\|z_{i}^{k,t+1}-\proj{x^k}\|^2\\
                = \ \ &\left\|\proj{\proj{x^k}-\eta \sum_{l=0}^t(G_{i}^{k,l} +c_i^k)}-\tilde{x}^{k+1}+\tilde{x}^{k+1}-\proj{x^k}\right\|^2\\
                \le \ \ & \underbrace{2\E\left\|\proj{\proj{x^k}-\eta \sum_{l=0}^t(G_{i}^{k,l} +c_i^k)}-\tilde{x}^{k+1}\right\|^2}\limits_{(a_1)}+\underbrace{2\E\|\tilde{x}^{k+1}-\proj{x^k}\|^2}\limits_{(a_2)}
            \end{split}
        \end{equation}
        For the term $(a_2)$, due to equation \eqref{eq.project Lips}, we can get
        \begin{equation}\label{eq.(a2)}
            (a_2)\le 2M^2\tau^2\eta^2\|\grad f(\proj{x^k})\|^2
        \end{equation}
        For the term $(a_1)$, due to the equivalent form of $c_i^k$ in Lemma \ref{lemma.equivalent update} and Lemma \ref{lemma.uniform boundness}, we have 
        $$
            \left\|\sum_{l=0}^t(G_{i}^{k,l} +c_i^k))\right\|\le 3\tau \chi_G.
        $$
        Since the step size satisfies $\tilde{\eta}\le \min\left\{\frac{\eta_g\gamma}{\chi},\frac{\eta_g\gamma}{3\chi_G}\right\}$, then applying Assumption \ref{assum.Lipschitz} and Lemma \ref{lemma.uniform boundness} implies
        $$
            \proj{x^k}-\eta \sum_{l=0}^t(G_{i}^{k,l} +c_i^k),\proj{x^k}-(t+1)\eta\grad f(\proj{x^k})\in \overline{U}_{\cM}(\gamma).
        $$
        Now we may use the property of proximal smoothness to get
        \begin{equation}\label{eq.(a1)middle}
            \begin{split}
                (a_1) \ \le \ \ &4\E\left\|\eta\sum_{l=0}^t\left(G_{i}^{k,l} +c_i^k-\grad f(\proj{x^k})\right)\right\|^2\\
                = \ \ &4\E\Bigg\|\eta\sum_{l=0}^t\Big(G_{i}^{k,l} - \grad f_i(\proj{x^k}) + \grad f_i(\proj{x^k})\\
                \ \ &+ \frac{1}{\tau}\sum_{t=0}^{\tau -1}\frac{1}{n}\sum_{i=1}^nG_{i}^{k-1,t}-\frac{1}{\tau}\sum_{t=0}^{\tau -1} G_{i}^{k-1,t}-\grad f(\proj{x^k})\Big)\Bigg\|^2\\
                =\ \ &4\E\left\|\eta\sum_{l=0}^t\left(G_{i}^{k,l}-\grad f_i(\proj{x^k})+\frac{1}{\eta\tau}(\Lambda_i^k-\overline{\Lambda}^k)\right)\right\|^2\\
                \le \ \ & \underbrace{8\E\left\|\eta\sum_{l=0}^t\left(G_{i}^{k,l}-\grad f_i(\proj{x^k})\right)\right\|^2}_{(b)}+8\left(\frac{t+1}{\tau}\right)^2\E\left\|\Lambda_i^k-\overline{\Lambda}^k\right\|^2
            \end{split}
        \end{equation}
        where the first equation is from the equivalent form of $c_i^k$ in Lemma \ref{lemma.equivalent update}, and the second equation is due to equations \eqref{eq.Lambda^k} and \eqref{eq.ave of Lambda^k}.
        From triangle inequality and Lemma \ref{lemma.lispchitz continuity over Riemannian gradient}, we have
        \begin{equation}\label{eq.(b) form}
            \begin{split}
                (b) \ = \ \ &8\E\left\|\eta\sum_{l=0}^t\left(G_{i}^{k,l}-\grad f_i(z_{i}^{k,l})+\grad f_i(z_{i}^{k,l})-\grad f_i(\proj{x^k})\right)\right\|^2\\
                \le \ \ &16\eta^2\E\left\|\sum_{l=0}^t\left(G_{i}^{k,l}-\grad f_i(z_{i}^{k,l})\right)\right\|^2 + 16\eta^2\E\left\|\sum_{l=0}^t\left(\grad f_i(z_{i}^{k,l})-\grad f_i(\proj{x^k})\right)\right\|^2\\
                \le \ \ &\underbrace{16\eta^2\E\left\|\sum_{l=0}^t\left(G_{i}^{k,l}-\grad f_i(z_{i}^{k,l})\right)\right\|^2}_{(c)} + 16\eta^2L^2(t+1)\sum_{l=0}^t\|z_{i}^{k,l}-\proj{x^k}\|^2.
            \end{split}
        \end{equation}
        Now we investigate the term (c). Notice that
        \begin{align}\label{eq.(c) form}
            (c) = 16\eta^2\sum_{l=0}^t\E\left\|G_{i}^{k,l}-\grad f_i(z_{i}^{k,l})\right\|^2 + 32\eta^2\sum_{0\le l < \omega \le t} \E\left[\inner{G_{i}^{k,l} - \grad f_i(z_{i}^{k,l})}{G_{i}^{k,\omega} - \grad f_i(z_{i}^{k,\omega})}\right].
        \end{align}
        For the inner product, let $U_{i}^{k,\omega} = \{u_{i,1}^{k,\omega},\ldots,u_{i,m}^{k,\omega}\}$ and $\Xi_{i}^{k,\omega} = \{\xi_{i,1}^{k,\omega},\ldots,\xi_{i,m}^{k,\omega}\}$, by the tower property of expectation, we get 
        \begin{align*}
           \E\left[\inner{G_{i}^{k,l} - \grad f_i(z_{i}^{k,l})}{G_{i}^{k,\omega} - \grad f_i(z_{i}^{k,\omega})}\right]= \ \ &\E\left[\E_{U_{i}^{k,\omega},\Xi_{i}^{k,\omega}}\left[\inner{G_{i}^{k,l} - \grad f_i(z_{i}^{k,l})}{G_{i}^{k,\omega} - \grad f_i(z_{i}^{k,\omega})}\mid z_{i}^{k,\omega}\right] \right] \\
           \le  \ \ &\left(\chi_G + \chi\right) \E\left[ \left\|\E_{U_{i}^{k,\omega},\Xi_{i}^{k,\omega}}\left[G_{i}^{k,\omega} - \grad f_i(z_{i}^{k,\omega})\mid z_{i}^{k,\omega}\right]\right\|\right]  \\ 
           \le \ \  &\left(\chi_G + \chi\right) \E\left[\chi_f\mu\right]  \\ 
           =  \ \ &\left(\chi_G + \chi\right)\chi_f\mu,
        \end{align*}
        where the first inequality is due to Assumption \ref{assum.Lipschitz}, Lemma \ref{lemma.uniform boundness} and the second inequality is from Lemma \ref{lemma.first-order moment}. By substituting the above inequality into \eqref{eq.(c) form} and using Lemma \ref{lemma.second-order moment}, it follows
        \begin{equation}\label{eq.(c)}
            \begin{split}
                (c) \ \le  \ \ 
            &16\eta^2\sum_{l=0}^t\E\left\|G_{i}^{k,l}-\grad f_i(z_{i}^{k,l})\right\|^2 + 16\eta^2t(t+1)\left(\chi_G + \chi\right)\chi_f\mu\\
            \le \ \ &  16\eta^2(t+1)(\chi_1\mu^2+\chi_2/m) + 16\eta^2t(t+1)\left(\chi_G + \chi\right)\chi_f\mu.
            \end{split}
        \end{equation}
        Combining all the inequalities \eqref{eq.(c)}, \eqref{eq.(b) form}, \eqref{eq.(a1)middle}, \eqref{eq.(a2)} and \eqref{eq.(z-p)1} implies
        \begin{equation}\label{eq.(z-p)2}
            \begin{split}
                &\E\|z_{i}^{k,t+1}-\proj{x^k}\|^2\\
                \le \ \ &16\eta^2L^2(t+1)\sum_{l=0}^t\|z_{i}^{k,l}-\proj{x^k}\|^2\\
                & +2M^2\tau^2\eta^2\|\grad f(\proj{x^k})\|^2+8\E\left\|\Lambda_i^k-\overline{\Lambda}^k\right\|^2\\
                & + 16\eta^2\tau(\chi_1\mu^2+\chi_2/m)+16\eta^2\tau(\tau-1)\left(\chi_G + \chi\right)\chi_f\mu.
            \end{split}
        \end{equation}
        Denote the sum of the last four terms as $h^k$:
        $$
            h^k\triangleq 2M^2\tau^2\eta^2\|\grad f(\proj{x^k})\|^2+8\E\left\|\Lambda_i^k-\overline{\Lambda}^k\right\|^2+ 16\eta^2\tau(\chi_1\mu^2+\chi_2/m)+16\eta^2\tau(\tau-1)\left(\chi_G + \chi\right)\chi_f\mu.
        $$
        Additionally, define
        $$
            S_{i,t}\triangleq\sum_{l=0}^t\E\|z_{i}^{k,t}-\proj{x^k}\|^2.
        $$
        From equation (\ref{eq.(z-p)2}) and the condition $\tilde{\eta}\le \eta_g/(16L)$, it follows that
        \begin{equation*}
            S_{i,t+1}^k\le \left(1+\frac{1}{(16\tau)}\right)S_{i,t}^k+h^k.
        \end{equation*}
        By further expanding this recurrence, we obtain
        \begin{equation}\label{eq.S-rec}
            S_{i,\tau-1}^k\le h^k\sum_{l=0}^{\tau-2}\left(1+1/(16\tau)\right)^l\le 1.1\tau h^k,
        \end{equation}
        where the last inequality is from $\sum_{l=0}^{\tau-2}(1+1/(16\tau))^l\le \sum_{l=0}^{\tau-2}{\rm exp}(l/(16\tau))\le \sum_{l=0}^{\tau-2}{\rm exp}(1/(16))\le 1.1\tau$. 
        
        Summing (\ref{eq.S-rec}) over all clients $i$ gives
        \begin{equation}\label{eq.(z-p)sum}
            \begin{aligned}
              \E\left[\sum_{i = 1}^n\sum_{t=0}^{\tau-1}\|z_{i}^{k,t}-\proj{x^k}\|^2\right]\ \le \ \ &3nM^2\tau^3\eta^2\|\grad f(\proj{x^k})\|^2+9\tau\E\left\|\blambda_i^k-\overline{\blambda}^k\right\|^2\\
            &\ \ + 18n\eta^2\tau^2(\chi_1\mu^2+\chi_2/m)+18n\eta^2\tau^2(\tau-1)\left(\chi_G + \chi\right)\chi_f\mu.  
            \end{aligned}
        \end{equation}
       Recall the definitions of $\blambda_i^k$ and $\overline{\blambda}^k$, and it follows
        \begin{equation}
            \begin{split}
                &\E\|\blambda_i^{k+1}-\overline{\blambda}^{k+1}\|^2\\
                = \ \ &\eta^2\E\left\|\tau\grad\f(\proj{\bx^{k+1}})-\sum_{t=0}^{\tau-1}\bG^{k,t}-\tau\overline{\grad\f}(\proj{x^{k+1}})+\sum_{t=0}^{\tau-1}\overline{\bG}^{k,t}\right\|^2\\
                \le \ \ &\eta^2\E\left\|\tau\grad\f(\proj{\bx^{k+1}})-\sum_{t=0}^{\tau-1}\bG^{k,t}\right\|^2\\
                =\ \ &\eta^2\E\Bigg\|\tau\grad\f(\proj{\bx^{k+1}})-\tau\grad\f(\proj{\bx^{k}})+\tau\grad\f(\proj{\bx^{k}})-\sum_{t=0}^{\tau-1}\grad\f(\bz_t^k)\\
                &\ \ +\sum_{t=0}^{\tau-1}\grad\f(\bz_t^k)-\sum_{t=0}^{\tau-1}\bG^{k,t}\Bigg\|\\
                \le &\ \ 2n\eta^2L^2\tau^2\E\|\proj{\bx^{k+1}}-\proj{\bx^{k}}\|^2+4\eta^2L^2\tau\sum_{i = 1}^n\sum_{t=0}^{\tau-1}\E\left[\|z_{i}^{k,t}-\proj{x^k}\|^2\right]\\
                & \ \ +4n\eta^2\tau(\chi_1\mu^2+\chi_2/m)+4n\eta^2\tau(\tau-1)\left(\chi_G + \chi\right)\chi_f\mu.
            \end{split}
        \end{equation}
        The first inequality holds since for any  $x_1, \dots, x_n \in \mathbb{R}^{d\times r}$, we have
$$
\sum_{i=1}^n\|x_i-\bar{x}\|^2\le \sum_{i=1}^n\|x_i\|^2, \quad \text{where} \quad \bar{x}=\frac{1}{n}\sum_{i=1}^n x_i.
$$
The final inequality follows from property \eqref{eq.grad Lips}, Lemma \ref{lemma.first-order moment}, Lemma \ref{lemma.second-order moment}, and a similar argument used in proving inequality \eqref{eq.(c)}. Finally, Substituting (\ref{eq.(z-p)sum}) into the above inequality completes the proof.
    \end{proof}

     \begin{lemma}
         Suppose Assumptions \ref{assum.proximal-smooth}, \ref{assum.unbiased}, and \ref{assum.Lipschitz} hold. If the step size satisfies $\tilde{\eta}\triangleq \eta_g\eta\tau\le \min\left\{\frac{\eta_g\gamma}{\chi},\frac{\eta_g\gamma}{3\chi_G},\frac{\eta_g}{16L}\right\}$, then we have
         \begin{equation}\label{eq.first term bound}
            \begin{split}
                \E[f(\proj{x^{k+1}}] \ \le \ \ &\E\Bigg[f(\proj{x^{k}})+\left(L-\frac{1}{2\tilde{\eta}}+\frac{\rho}{2}\right)\|\tilde{x}^{k+1}-\proj{x^k}\|^2\\
                &+\left(\frac{L}{2}-\frac{1}{2\tilde{\eta}}\right)\|\proj{x^{k+1}}-\proj{x^k}\|^2\\
                &+\frac{\tilde{\eta}}{2(1-\tilde{\eta}\rho)}\frac{2L^2}{n\tau}\Big(3nM^2\tau^3\eta^2\|\grad f(\proj{x^k})\|^2+9\tau\E\left\|\blambda_i^k-\overline{\blambda}^k\right\|^2\\
            &+ 18n\eta^2\tau^2(\chi_1\mu^2+\chi_2/m)+18n\eta^2\tau^2(\tau-1)\left(\chi_G + \chi\right)\chi_f\mu\Big)\\
                &+\frac{\tilde{\eta}}{2(1-\tilde{\eta}\rho)}\left(\frac{2}{n\tau}(\chi_1\mu^2+\chi_2/m) + 2\left(\chi_G + \chi\right)\chi_f\mu\right)\Bigg].
            \end{split}
         \end{equation}
     \end{lemma}
     \begin{proof}
         Let $x^+ = \tilde{x}^{k+1} = \proj{\proj{x^k}-(t+1)\eta\grad f(\proj{x^k})}$, $z = \proj{x^k}$, $x = \proj{x^k}$, $v = \grad f(\proj{x^k})$, and $\rho = \frac{3D}{2\gamma}$. By applying Lemma \ref{lemma.special} to (\ref{eq.PRGD}), it follows
         \begin{equation}
             \E[f(\tilde{x}^{k+1})]\le \E\left[ f(\proj{x^k})+\left(\frac{L}{2}-\frac{1}{2\eta}\right)\|\tilde{x}^{k+1}-\proj{(x^k)}\|^2-\frac{1-\tilde{\eta}\rho}{2\tilde{\eta}}\|\tilde{x}^{k+1}-\proj{x^k}\|^2\right],
         \end{equation}
         where we use the condition $\tilde{\eta} \leq \frac{\eta_g\gamma}{\chi}$ to ensure that $\proj{x^k}-(t+1)\eta\grad f(\proj{x^k}) \in \overline{U}_{\cM}(\gamma)$. From Lemma \ref{lemma.equivalent update} and the condition $\tilde{\eta} \leq \frac{\eta_g\gamma}{3\chi_G}$, we obtain
         \begin{equation}\label{eq.use1}
             \proj{x^{k+1}}=\proj{\proj{x^k}-\tilde{\eta}\frac{1}{n\tau}\sum_{i=1}^n\sum_{t = 0}^{\tau-1}G_{i}^{k,t}}\in \overline{U}_{\cM}(\gamma).
         \end{equation}
         Similarly, let $x^+=\proj{x^{k+1}}$, $x=  \proj{x^k}$, $z=\tilde{x}^{k+1}$, $v = v^k = \tilde{\eta}\frac{1}{n\tau}\sum_{i=1}^n\sum_{t = 0}^{\tau-1}G_{i}^{k,t}$. Hence using Lemma \ref{lemma.special} implies
         \begin{equation}\label{eq.use2}
            \begin{split}
                &\E[f(\proj{x^{k+1}})\\
                \le \ \ &\E\Big[f(\tilde{x}^{k+1})+\inner{\grad f(\proj{x^k})-v^k}{\proj{x^{k+1}}-\tilde{x}^{k+1}}\\
                &-\frac{1}{2\tilde{\eta}}(\|\proj{x^{k+1}}-\proj{x^k}\|^2-\|\tilde{x}^{k+1}-\proj{x^k}\|^2)-\frac{1-\tilde{\eta}\rho}{\tilde{\eta}\rho}\|\tilde{x}^{k+1}-\proj{x^{k+1}}\|^2\\
                &+\frac{L}{2}\|\proj{x^{k+1}}-\proj{x^k}\|^2+\frac{L}{2}\|\tilde{x}^{k+1}-\proj{x^k}\|^2\Big]\\
                =\ \ &\E\Big[f(\tilde{x}^{k+1})+\inner{\grad f(\proj{x^k})-v^k}{\proj{x^{k+1}}-\tilde{x}^{k+1}} + \left(\frac{L}{2}-\frac{1}{2\tilde{\eta}}\right)\|\proj{x^{k+1}}-\proj{x^k}\|^2\\
                &+\left(\frac{L}{2}+\frac{1}{2\tilde{\eta}}\right)\|\tilde{x}^{k+1}-\proj{x^k}\|^2-\frac{1-\tilde{\eta}\rho}{2\tilde{\eta}}\|\tilde{x}^{k+1}-\proj{x^{k+1}}\|^2\Big].
            \end{split}
         \end{equation}
         Summing equations \eqref{eq.use1} and \eqref{eq.use2} yields
         \begin{equation}\label{eq.f(P)1}
            \begin{split}
                \E[f(\proj{x^{k+1}}]\le \ \ &\E\Big[f(\proj{x^{k}})+\left(L-\frac{1}{2\tilde{\eta}}+\frac{\rho}{2}\right)\|\tilde{x}^{k+1}-\proj{x^k}\|^2\\
                &+\left(\frac{L}{2}-\frac{1}{2\tilde{\eta}}\right)\|\proj{x^{k+1}}-\proj{x^k}\|^2-\frac{1-\tilde{\eta}\rho}{2\tilde{\eta}}\|\tilde{x}^{k+1}-\proj{x^{k+1}}\|^2 \\
                &+\inner{\grad f(\proj{x^k})-v^k}{\proj{x^{k+1}}-\tilde{x}^{k+1}}\Big] \\
                \le \ \ &\E\Big[f(\proj{x^{k}})+\left(L-\frac{1}{2\tilde{\eta}}+\frac{\rho}{2}\right)\|\tilde{x}^{k+1}-\proj{x^k}\|^2\\
                &+\left(\frac{L}{2}-\frac{1}{2\tilde{\eta}}\right)\|\proj{x^{k+1}}-\proj{x^k}\|^2-\frac{1-\tilde{\eta}\rho}{2\tilde{\eta}}\|\tilde{x}^{k+1}-\proj{x^{k+1}}\|^2\\
                &+\frac{1-\tilde{\eta}\rho}{2\tilde{\eta}}\|\proj{x^{k+1}}-\tilde{x}^{k+1}\|^2+\frac{\tilde{\eta}}{2(1-\tilde{\eta}\rho)}\|\grad f(\proj{x^k})-v^k\|^2\Big] \\
                = \ \ &\E\Big[f(\proj{x^{k}})+\left(L-\frac{1}{2\tilde{\eta}}+\frac{\rho}{2}\right)\|\tilde{x}^{k+1}-\proj{x^k}\|^2\\
                &+\left(\frac{L}{2}-\frac{1}{2\tilde{\eta}}\right)\|\proj{x^{k+1}}-\proj{x^k}\|^2 + \frac{\tilde{\eta}}{2(1-\tilde{\eta}\rho)}\|\grad f(\proj{x^k})-v^k\|^2\Big],
            \end{split}
         \end{equation}
         where we use Lemma \ref{lemma.cauchy schwarz} in the second inequality. To proceed, we bound the last term as follows:
         \begin{equation}\label{eq.gradf-v}
            \begin{split}
                 &\E\|\grad f(\proj{x^k})-v^k\|^2\\
                 = \ \ &\E\left\|\frac{1}{n\tau}\sum_{i=1}^n\sum_{t=0}^{\tau-1}\left(G_{i}^{k,t}-\grad f_i(z_{i}^{k,t})+\grad f_i(z_{i}^{k,t})-\grad f_i(\proj{x^k})\right)\right\|^2\\
                 \le \ \  &\frac{2L^2}{n\tau}\sum_{i=1}^n\sum_{t=0}^{\tau-1}\|z_{i}^{k,t}-\proj{x^k}\|^2 + \frac{2}{n\tau}(\chi_1\mu^2+\chi_2/m) + 2\left(\chi_G + \chi\right)\chi_f\mu,
            \end{split}
         \end{equation}
        where the last inequality follows from property \eqref{eq.grad Lips}, Lemma \ref{lemma.first-order moment}, Lemma \ref{lemma.second-order moment}, and a similar argument used in proving inequality \eqref{eq.(c)}. By combining equations \eqref{eq.f(P)1} and \eqref{eq.gradf-v}, we get 
         \begin{equation}\label{eq.f(P)2}
            \begin{split}
                \E[f(\proj{x^{k+1}}] \ \le \ \ &\E\Bigg[f(\proj{x^{k}})+\left(L-\frac{1}{2\tilde{\eta}}+\frac{\rho}{2}\right)\|\tilde{x}^{k+1}-\proj{x^k}\|^2\\
                &+\left(\frac{L}{2}-\frac{1}{2\tilde{\eta}}\right)\|\proj{x^{k+1}}-\proj{x^k}\|^2\\
                &+\frac{\tilde{\eta}}{2(1-\tilde{\eta}\rho)}\left(\frac{2L^2}{n\tau}\sum_{i=1}^n\sum_{t=0}^{\tau-1}\|z_{i}^{k,t}-\proj{x^k}\|^2 + \frac{2}{n\tau}(\chi_1\mu^2+\chi_2/m) + 2\left(\chi_G + \chi\right)\chi_f\mu\right)\Bigg].
            \end{split}
         \end{equation}
         Substituting equation \eqref{eq.(z-p)sum} into the above inequality completes the proof.
     \end{proof}
     \subsection{Proof of Theorem \ref{theorem.convergence}:}
    \begin{theorem}[Restatement]
    Suppose Assumptions \ref{assum.proximal-smooth}, \ref{assum.unbiased} and \ref{assum.Lipschitz} hold. If the step sizes satisfy $\eta_g = \sqrt{n}$ and
    \begin{align*}
       \tilde{\eta}\triangleq \eta_g\eta\tau\le \min\left\{\frac{1}{24ML},\frac{\gamma}{6\max\{\chi_G,\chi\}},\frac{1}{\chi L_{\cP}}\right\}, 
    \end{align*}
    then we have 
    \begin{align*}
        \frac{1}
        {K}\sum_{k=1}^K\|\cG_{\tilde{\eta}}(\proj{x^k})\|^2 \le \frac{8\Omega^1}{\sqrt{n}\eta\tau K}+\frac{64}{n\tau}\left(\chi_1\mu^2+\frac{\chi_2}{m}\right) +\frac{16(3+n)\left(\chi_G + \chi\right)\chi_f}{n}\mu,    
    \end{align*}
     where $\Omega^1>0$ is a constant related to initialization. Specifically, set the smoothing parameter $\mu = \mathcal{O}(1/n\tau K)$, and it follows
     \begin{align*}
        \frac{1}
        {K}\sum_{k=1}^K\|\cG_{\tilde{\eta}}(\proj{x^k})\|^2 = \mathcal{O}\left(\frac{1}{\sqrt{n} \tau K} + \frac{1}{n\tau m}\right).
     \end{align*}
\end{theorem}
     \begin{proof}
         By combining equations \eqref{eq.first term bound} and \eqref{eq.second term bound}, it implies
         \begin{equation}
             \begin{split}
                 &\E\left[(f(\proj{x^{k+1}})-f^*+\frac{1}{\tilde{\eta}n}\|\blambda^{k+1}-\overline{\blambda}^{k+1}\|^2)\right]\\
                 \le \ \ &\E\left[(f(\proj{x^{k}})-f^*)+\frac{1}{\tilde{\eta}n}\|\blambda^{k}-\overline{\blambda}^{k}\|^2-\frac{\tilde{\eta}}{8}\|\cG_{\tilde{\eta}}(\proj{x^k})\|^2\right]\\
                 &+ \frac{8\tilde{\eta}}{n\tau}(\chi_1\mu^2+\chi_2/m)+\frac{(6+2n)\tilde{\eta}}{n}\left(\chi_G + \chi\right)\chi_f\mu,
             \end{split}
         \end{equation}
         where we use the inequality $\|\grad f(\cP_{\cM}(x^k))\|\leq 2\|\cG_{\tilde{\eta}}(\cP_{\cM}(x^k))\|$ \citep[Lemma A.2]{zhang2024nonconvex}, along with the condition $\tilde{\eta}\le \min\left\{\frac{1}{24ML},\frac{\gamma}{6\max\{\chi_G,\chi\}},\frac{1}{\chi L_{\cP}}\right\}$, and the definition of $\cG_{\tilde{\eta}}(\proj{x^k})$. Lengthy and nonessential algebraic manipulations are omitted for brevity. By using the definition of the auxiliary function $\Omega^k$ and rearranging the above inequality, we obtain
         \begin{equation}
             \frac{\tilde{\eta}}{8}\E\|\cG_{\tilde{\eta}}(\proj{x^k})\|^2\le \E[\Omega^{k}] - \E[\Omega^{k+1}] + \frac{8\tilde{\eta}}{n\tau}(\chi_1\mu^2+\chi_2/m)+\frac{(6+2n)\tilde{\eta}}{n}\left(\chi_G + \chi\right)\chi_f\mu.
         \end{equation}
         Summing over $k$ from $1$ to $K$ completes the proof.
         

     \end{proof}

\end{document}